\documentclass[11pt,reqno]{amsart}
\usepackage{xifthen,setspace}
\usepackage{bbm}

\usepackage[comma,sort&compress, numbers]{natbib}
\RequirePackage{hypernat}
\usepackage{amssymb,tabularx,multicol,multirow,booktabs,csquotes}
\usepackage[top=1in, bottom=1in, left=1in, right=1in]{geometry}
\usepackage{fancyhdr,lipsum, makecell}
\usepackage{booktabs, makecell}
\usepackage[space]{grffile}
\usepackage{lineno}

\usepackage{pkgfile}

\allowdisplaybreaks
\theoremstyle{definition}

\title[Sparse Uniformity Testing]{Sparse Uniformity Testing}

\author{Bhaswar B. Bhattacharya}
\address{Department of Statistics and Data Science, University of Pennsylvania, Philadelphia, USA} 
\email{bhaswar@wharton.upenn.edu}

\author{Rajarshi Mukherjee}
\address{Department of Biostatistics, Harvard University, Boston, USA} 
\email{rmukherj@hsph.harvard.edu} 

\begin{document}

\begin{abstract} 	
In this paper we consider the uniformity testing problem for high-dimensional discrete distributions (multinomials) under sparse alternatives. More precisely, we derive sharp detection thresholds for testing, based on $n$ samples, whether a discrete distribution supported on $d$ elements differs from the uniform distribution only in $s$ (out of the $d$) coordinates and is $\varepsilon$-far (in total variation distance) from  uniformity. Our results reveal various interesting phase transitions which depend on the interplay of the sample size $n$ and the signal strength $\varepsilon$ with the dimension $d$ and the sparsity level $s$. For instance, if the sample size is less than a threshold (which depends on $d$ and $s$), then all tests are asymptotically powerless, irrespective of the magnitude of the signal strength. On the other hand, if the sample size is above the threshold, then the detection boundary undergoes a further phase transition depending on the signal strength. Here, a $\chi^2$-type test attains the detection boundary in the dense regime, whereas in the sparse regime a Bonferroni correction of two maximum-type tests and a version of the Higher Criticism test is optimal up to sharp constants. These results combined provide a complete description of the phase diagram for the sparse uniformity testing problem across all regimes of the parameters $n$, $d$, $s$, and $\varepsilon$. One of the challenges in dealing with multinomials is that the parameters are always constrained to lie in the simplex. This results in the aforementioned two-layered phase transition, a new phenomenon which does not arise in classical high-dimensional sparse testing problems.   
\end{abstract}
		

	
	\maketitle

	\section{Introduction}\label{sec:intro}

Testing whether independent samples from an unknown probability distribution are uniformly distributed over a domain is a classical problem in statistical inference. In this paper we consider the {\it uniformity testing problem} for discrete distributions (multinomials) where, given independent and identically distributed samples  $X_1,\ldots,X_n$ from an unknown probability distribution $\bm p$ on a discrete domain $[d]:=\{1, 2, \ldots, d\}$, the goal is to determine whether $\bm p$ is uniformly distributed over $[d]$ or whether $\bm p$ is $\varepsilon$-far (in total variation distance) from the uniform distribution on $[d]$. The minimax detection radius for this problem, that is, the threshold of $\varepsilon$ above which there exists an asymptotically powerful test and below which all tests are asymptotically powerless, was obtained in the seminal  paper of Paninsky \cite{paninski2008coincidence} (see also \cite{goldreich2011testing} where the problem arose in the context of testing expansion properties of graphs), which was followed by a slew of refinements and extensions (see \cite{ADKoptimal,balakrishnan2019hypothesis,butucea2021locally,canonneuniformity,chhor2022robust,chhor2020sharp,DKuniformity,valiant2017automatic} and the references therein). Recent surveys of Balakrishnan and Wasserman \cite{balakrishnan2018hypothesis} and Canonne \cite{canonne2020survey} provide excellent reviews of the main results in this direction. 

In many applications, the distribution $\bm p$ differs from the null distribution in certain structured ways.  A primary example of such structured hypotheses is the case of sparsity recently considered by Donoho and Kipnis  \cite{donoho2020two}, Kipnis \cite{kipnis2019higher,kipnis2021unification}, and Kipnis and Donoho \cite{kipnis2021inability} in the context of two sample testing of discrete distributions. In the context of uniformity testing on $d$-alphabets, this amounts to adding a sparsity constraint to control the number of coordinates where the allowed distributions are different from the uniform case. In this paper we consider this problem and completely characterize the \textit{sharp} asymptotic minimax detection boundary across all regimes of sparsity, dimension, and sample sizes. 
To formally state our problem, denote by $\cP([d])$ the collection of all probability distributions on the set $[d] := \{1, 2, \ldots, d\}$. More precisely, 
$$\cP([d]) := \left\{ \bm p = (p_1, p_2, \ldots, p_d) : \text{ such that } 0 \leq p_j \leq 1, \text{ for } 1 \leq j \leq d, \text{ and } \sum_{j=1}^d p_j= 1 \right\} .$$ Moreover, let us denote  by $U([d])$ the uniform distribution on $[d]$, where $p_j=1/d$, for all $1 \leq j \leq d$. Subsequently, given a signal strength parameter $\varepsilon > 0$ and a sparsity parameter $s = d^{1-\alpha}$, where $0 \leq \alpha \leq 1$, we formalize the {\it sparse uniformity testing problem} through the following hypotheses: Given i.i.d. samples $X_1,\ldots,X_n$ from an unknown probability distribution $\bm p \in \cP([d])$ test 
	\begin{align}\label{eq:uniform_H} 
	H_0: \bm{p}= U([d]) \quad \text{versus} \quad H_1: \bm{p}\in \mathcal{P}(U[d],s,\varepsilon),
	\end{align} 
	where\footnote{For any vector $\bm x =(x_1, x_2, \ldots, x_d) \in \R^d$, $||\bm x||_1:=\sum_{j=1}^d |x_j|$ and $||\bm x||_2 := \sqrt{\sum_{j=1}^d x_j^2}$ denote the $L_1$ and $L_2$ norms of $\bm x$, respectively.} 
	$$\mathcal{P}(U([d]), s, \varepsilon)=\left\{ \bm{p}\in \cP([d]): \|\bm{p}-U([d])\|_0= s, \text{ and }\|\bm{p}-U([d])\|_1\geq \varepsilon \right\}.$$ 
	Note that $\mathcal{P}(U([d]),s,\varepsilon)$ is the collection of multinomial distributions which differ from the uniform distribution $U([d])$ in $s$-coordinates ($s$-sparse) and are $\varepsilon$-far from $U([d])$ in the $L_1$-distance. It is important to note that the testing problem \eqref{eq:uniform_H} makes sense only when the set $\mathcal{P}(U([d]),s,\varepsilon)$ non-empty. In fact, it is easy to find the maximum attainable value of $\varepsilon$ for which 
$\mathcal{P}(U([d]),s,\varepsilon)$ is non-empty. We summarize this in the following proposition for ease of referencing in the rest of the paper. The proof is given in Appendix \ref{sec:lm_pf}. (Throughout, unless stated otherwise, all asymptotic limits should be thought of as $d \rightarrow \infty$). 
 
 \begin{ppn}\label{ppn:L01} Fix a sparsity level $s=d^{1-\alpha}$. Then $$\lim \frac{d \varepsilon_{\max}}{s} = 2,$$
where $\varepsilon_{\max}:=\sup_{\bm{p}\in \cP([d]): \|\bm{p}-U([d])\|_0= s } ||\bm p - U([d])||_1$.  
\end{ppn}


Given i.i.d. samples $X_1, X_2, \ldots, X_n$ define its {\it histogram} $\bm Z= (Z_1, Z_2, \ldots, Z_d)$ as follows: 
$$Z_j= \sum_{i=1}^n \bm 1\{X_i=j\}.$$	
Note that $Z_j$ counts the number of occurrences of $j \in [d]$ in the sample $X_1, X_2, \ldots, X_n$, and $\bm Z= (Z_1, Z_2, \ldots, Z_d)$ is a sufficient statistic for the distribution $\bm p= (p_1, p_2, \ldots, p_d)$. Clearly, $\bm Z \sim \mathrm{Multi}(p_1, p_2, \ldots, p_d)$, the multinomial distribution with parameters $(p_1, p_2, \ldots, p_d)$. Now, to establish a decision theoretic formalism for the testing problem \eqref{eq:uniform_H} fix $\varepsilon > 0$ such that $\mathcal{P}(U([d]),s,\varepsilon)$ is non-empty. Then the worst case risk of a test function $T: \bm Z = (Z_1, Z_2, \ldots, Z_d) \rightarrow \{0,1\}$ is defined as: 
\begin{align}\label{eq:RT}
	\cR_{n, d}(T, s, \varepsilon):=\left\{ \mathbb{P}_{H_0}\left(T=1\right)+\sup_{\bm{p}\in \mathcal{P}(U([d]),s,\varepsilon)}\mathbb{P}_{\bm{p}}\left(T=0\right) \right\}, 
\end{align}	
	where 
\begin{itemize}	
\item[--] $\mathbb{P}_{H_0}$ denotes the probability under the null (which is the uniform distribution on $[d]$), where $\bm Z \sim \mathrm{Multi}(n/d, n/d, \ldots, n/d)$, and 

\item[--] $\mathbb{P}_{\bm{p}}$ denotes the probability distribution corresponding to  $\bm{p} = (p_1, p_2, \ldots, p_d) \in \mathcal{P}(U([d]),s,\varepsilon)$,  where $\bm Z=\mathrm{Multi}(p_1, p_2, \ldots, p_d)$. 
\end{itemize} 
Consequently, the minimax risk for the problem \eqref{eq:uniform_H} is given by: 
\begin{align}\label{eq:R}
	\cR_{n, d}(s, \varepsilon):=\inf_{T}\left\{ \mathbb{P}_{H_0}\left(T=1\right)+\sup_{\bm{p}\in \mathcal{P}(U([d]),s,\varepsilon)}\mathbb{P}_{\bm{p}}\left(T=0\right) \right\}, 
\end{align}	
where the infimum in \eqref{eq:R} is over all test functions $T: \bm Z \rightarrow \{0,1\}$. A sequence of test functions $T_n$ is said to be {\it asymptotically powerful} for the problem \eqref{eq:uniform_H} if $\lim \cR(T_n, n, d, s, \varepsilon) = 0$. On the other hand, a sequence of test functions $T_n$ is said to be {\it asymptotically powerless} for \eqref{eq:uniform_H} if $\lim \cR(T_n, n, d, s, \varepsilon) \geq 1$.   Given $n, d, s$, the {\it critical signal strength} is the signal strength $\varepsilon_0$ for which $\mathcal{R}_{n,d}(s,\varepsilon)\rightarrow 1$  for $\varepsilon\ll \varepsilon_0 $ and $\mathcal{R}_{n,d}(s,\varepsilon)\rightarrow 0$ for $\varepsilon\gg \varepsilon_0$ (see Section \ref{sec:asymptoticnotation} for the formal definitions of the asymptotic notations).

In this paper, we derive the asymptotic minimax detection boundary  for the uniformity testing problem, that is, the asymptotic behavior of the critical signal strength as a function of $n, d, s$, for all values of $\alpha \in (0, 1)$. The following is a summary of the results obtained: 

	\begin{enumerate}
		\item [(1)] In the {\it dense regime}, that is, $0 \leq \alpha\leq \frac{1}{2}$, our results reveal two regimes of asymptotic behavior of the testing problem \eqref{eq:uniform_H} depending on the relative magnitude of $n$ and $d$. In particular, we show that when $n\gtrsim d^{\frac{1}{2}+\alpha}$ then the problem behaves similar to a  Gaussian sparse mean testing problem, and a $\chi^2$-type test is asymptotically optimal. In contrast, when $n\ll d^{\frac{1}{2}+\alpha}$, the behavior of the problem is much more delicate. Here, we show a sharp impossibility result in the sense that all tests are asymptotically powerless if $\limsup d\varepsilon/s \leq 2$. The result is sharp since $\lim d\varepsilon_{\max}/s =2$ (by Proposition \ref{ppn:L01}), which means that in the regime $n\ll d^{\frac{1}{2}+\alpha}$ all tests are powerless for any achievable signal strength.  The results in this regime are formalized in Theorem \ref{thm:one_sample_alphaleqhalf}. 
						
		\item [(2)] In the {\it sparse regime}, that is, $\frac{1}{2} < \alpha < 1$, there is once again a two-stage phase transition in the behavior of the testing problem \eqref{eq:uniform_H} depending on the relative magnitude of $n$ and $d$. In this case, for $n\gg d\log^3{d}$, we derive the optimal detection boundary up to sharp constants. Here, a sequence of sharp optimal tests is constructed by Bonferroni corrections of a Higher Criticism (HC)-type test and two versions of maximum-type tests. In contrast, when $n\ll d\log{d}$, all tests are powerless irrespective of the signal strength. The results in this regime are formalized in Theorem \ref{thm:one_sample_alphageqhalf}. 
	\end{enumerate}
	
The rest of the paper is organized as follows. The main results are described in Section \ref{sec:main_results}. The outline of the proofs are discussed in Section \ref{sec:outline_pf} and numerical results are presented in Section \ref{sec:numerical_results}. The proofs of the main results are given in Section \ref{sec:main_results_pf}. 

	\section{Main Results}\label{sec:main_results} 
	
In this section we will formally state our main theorems. Toward this, we begin by recalling few standard asymptotic notations. 	

\subsection{Asymptotic Notations}
\label{sec:asymptoticnotation} 
For positive sequences $\{a_n\}_{n\geq 1}$ and $\{b_n\}_{n\geq 1}$, $a_n \lesssim b_n$ means $a_n \leq C_1 b_n$, $a_n \gtrsim b_n$ means $a_n \geq C_2 b_n$, and $a_n \asymp b_n$ means $C_2 b_n \leq a_n \leq C_1 b_n$, for all $n$ large enough and positive constants $C_1, C_2$. 
Moreover, subscripts in the above notation,  for example  $\lesssim_\square$,  denote that the hidden constants may depend on the subscripted parameters. 
Finally, for positive sequences $\{a_n\}_{n\geq 1}$ and $\{b_n\}_{n\geq 1}$, $a_n \ll b_n$ means $a_n =o(b_n)$ and $a_n \gg b_n$ means $b_n  = o(b_n)$.

\subsection{Statements of the Results}

As discussed above our results have the following two regimes: (1) the dense regime where $0 \leq \alpha \leq \frac{1}{2}$ and (2) the sparse regime where $\frac{1}{2} < \alpha \leq 1$. We begin with the dense regime. 
	
		\begin{thm}[Dense regime] Fix a signal strength $\varepsilon > 0$, a sparsity level $s=d^{1-\alpha}$, with $0 \leq \alpha\leq \frac{1}{2}$, and define 
		\begin{align}\label{eq:threshold_epsilon}
		\varepsilon_1(n, d, s):=\left(\frac{s}{n\sqrt{d}} \right)^{\frac{1}{2}}.
		\end{align} 

	\begin{enumerate}
				\item[$(1)$] Suppose $n\gtrsim d^{\frac{1}{2}+\alpha}$. Then the following hold: 
		\begin{enumerate}
			\item [(a)] There is a sequence of asymptotically powerful tests if $\varepsilon \gg \varepsilon_1(n, d, s)$.
			
%
	
			\item [(b)] On the other hand, all tests are asymptotically powerless if $\varepsilon \ll \varepsilon_1(n, d, s)$.
		\end{enumerate} 
		
	\item[$(2)$] Suppose $n\ll d^{\frac{1}{2}+\alpha}$. Then  all tests are asymptotically powerless if $\limsup \frac{d\varepsilon}{s}\leq 2$. As a consequence, by Proposition $\ref{ppn:L01}$, in this regime all tests are asymptotically powerless for any asymptotically achievable signal strength.
	\end{enumerate}
\label{thm:one_sample_alphaleqhalf}
\end{thm}

The proof of this theorem is given in Section \ref{sec:main_results_pf} (an outline of the proof is given in Section \ref{thm:one_sample_alphaleqhalf}). Note that, as alluded to in the Introduction, the problem undergoes two phase transitions: If the sample size $n \ll d^{\frac{1}{2}+\alpha}$, then Theorem \ref{thm:one_sample_alphaleqhalf}~(2) shows that testing is impossible regardless of signal strength. On the other hand, for $n \gtrsim d^{\frac{1}{2}+\alpha}$, then there is a further threshold for the critical signal strength (given by \eqref{eq:threshold_epsilon}), below which testing is impossible and above which a  $\chi^2$-type test is optimal. While the latter parallels results on sparse testing in the Gaussian sequence model  (see \citep{donoho2004higher,tony2011optimal,arias2011global,ingster2010detection} and the references therein), the former regime where testing is impossible is a new phenomenon, which emerges from the constraint that the parameters of a multinomial belong to the simplex. 


Next, we consider the sparse regime $\frac{1}{2} < \alpha < 1$. To this end, define the function: 
 \begin{align}\label{eq:threshold_III}
 C(\alpha) := 
		\left\{\begin{array}{cc}
		(\alpha-\frac{1}{2})  &   \text{ if } \frac{1}{2} < \alpha < \frac{3}{4} ,   \\
		(1 - \sqrt{1-\alpha})^2  &    \alpha \geq \frac{3}{4} .
		\end{array}
		\right. 
		\end{align}
 The following result summarizes the minimax detection boundaries for the uniformity testing problem in the sparse regime.

\begin{thm}[Sparse regime] \label{thm:one_sample_alphageqhalf} 
 Fix a signal strength $\varepsilon > 0$, a sparsity level $s=d^{1-\alpha}$, with $\frac{1}{2} < \alpha < 1$, and define 
 \begin{align}\label{eq:threshold_II}
		\varepsilon_2(n, d, s):= s \left(\frac{2 \log{d}}{nd} \right)^{\frac{1}{2}}.
		\end{align} 

	\begin{enumerate}
		
		\item[$(1)$] Suppose $n\gg d\log^3{d}$. Then the following hold: 
		\begin{enumerate}
			\item [(a)] There exists a sequence of asymptotically powerful tests if
			\begin{align*}
			\liminf \frac{\varepsilon}{\varepsilon_2(n, d, s)}> \sqrt{C(\alpha)},
			\end{align*}
			where $C(\alpha)$ is as defined in \eqref{eq:threshold_II}. 
			
			\item [(b)] All tests are asymptotically powerless if
				\begin{align*}
			\limsup \frac{\varepsilon}{\varepsilon_2(n, d, s)}<  \sqrt{C(\alpha)}.
			\end{align*} 
		\end{enumerate}

\item[$(2)$] Suppose $n\ll d\log{d}$. Then  all tests are asymptotically powerless if $\limsup \frac{d\varepsilon}{s}\leq 2$. As a consequence, by Proposition $\ref{ppn:L01}$, in this regime all tests are asymptotically powerless for any asymptotically achievable signal strength.
		
	\end{enumerate} 
\end{thm}

The proof of this theorem is given in Section \ref{sec:main_results_pf} (an outline of the proof is given in Section \ref{thm:one_sample_alphageqhalf}). Similar to the dense regime, once again there is a two-fold phase transition. For small sample sizes ($n \ll d\log{d}$ in this case) testing is impossible regardless of the signal strength (Theorem \ref{thm:one_sample_alphageqhalf} (2)). On the other hand, for larger sample sizes ($n \gg d\log^3{d}$) the critical signal strength obtained in Theorem \ref{thm:one_sample_alphageqhalf} (1) is sharp up to the precise multiplicative constant. The test achieving this sharp optimality turns out to be a Bonferroni correction between 
two different maximum-type tests and a HC-type test. For other examples of HC-type tests in analogous sparse high-dimensional testing problems, see \citep{donoho2004higher,arias2011global,arias2015sparse,cai2014optimal,cai2016global,mukherjee2015hypothesis,mukherjee2018detection} and the references therein.

\begin{remark}[Closing the log-gap]\label{remark:log_gap} Note that there is a gap in a logarithmic factor (in the dimension $d$) between the upper and lower bounds on the sample size $n$ in Theorem \ref{thm:one_sample_alphageqhalf} (1) and Theorem \ref{thm:one_sample_alphageqhalf} (2), respectively. This is a technicality arising from our proof where we need the assumption $n\gg d\log^3 d$ to invoke a Cramer-type moderate deviation estimate for Poisson random variables. This estimate is required to establish the critical signal strength upto sharp constants as in Theorem \ref{thm:one_sample_alphageqhalf} (1). However, to derive only rate optimal results (that is, showing  consistent testing is possible if and only if $\varepsilon\gtrsim \varepsilon_2(n,d,s)$), this log-gap can be removed and $n\gg d\log{d}$ suffices. In this case, it can be shown that the upper bound is achieved by rejecting for large values of a maximum-type test (for example, the test based on the statistic $M_{n,d}$ defined in \eqref{eq:Mnd} and implemented as in Lemma \ref{lm:threshold_I} suffices) and the lower bound can be obtained by arguments similar to the proof of Theorem \ref{thm:one_sample_alphaleqhalf}. Moreover, to obtain only a rate optimal result one can avoid the construction of the modified prior and the truncated second moment method employed in the proof of the lower bound in Theorem \ref{thm:one_sample_alphageqhalf} (1) to obtain the sharp constant  (see Section \ref{sec:outline_pf} for details). 
\end{remark}

One immediate consequence of our results is that we recover the classical result of Paninsky \cite{paninski2008coincidence}, where the uniformity testing problem was studied without sparsity restrictions. In fact, we obtain more a granular exploration of the results in \cite{paninski2008coincidence}, owing our refined analyses based on different levels of sparsity. We elaborate on this in the following remark:

\begin{remark}(Implications to the classical uniformity testing problem) Paninsky \cite{paninski2008coincidence} derived the minimax separation rates for uniformity testing without assuming any sparsity. In particular, \cite{paninski2008coincidence} showed that without any sparsity constraints (that is, over the class $\bigcup_{s=1}^d \mathcal{P}(U([d]),s,\varepsilon)$) the minimax detection threshold for uniformly testing is 
\begin{align}\label{eq:thresholduniform}
\varepsilon_{\mathrm{unif}}(n, d) = \frac{d^{\frac{1}{4}}}{\sqrt n} . 
\end{align} 
Consequently, it is natural to compare this result with the worst case of the separations obtained in this paper. A simple comparison of the rates between the dense ($\alpha\leq \frac{1}{2}$) and sparse  ($\alpha> \frac{1}{2}$) cases from Theorem \ref{thm:one_sample_alphaleqhalf} and Theorem \ref{thm:one_sample_alphageqhalf} implies that the worst case between the two regimes is indeed obtained at $\alpha=0$ (the extreme of the dense case). Consequently, we clearly recover the threshold in \eqref{eq:thresholduniform} from \eqref{eq:threshold_epsilon} by substituting $\alpha = 0$. However, our results a say bit more when $n \lesssim d^{\frac{1}{2}+\alpha} = \sqrt d$. Note that in this regime $\varepsilon_{\mathrm{unif}}(n, d) \gtrsim 1$ whereas $\varepsilon_{\mathrm{max}} = 2$, therefore the results in \cite{paninski2008coincidence} are inapplicable when the signal strength $\varepsilon \asymp 1$ and $\varepsilon_{\mathrm{unif}}(n, d) \asymp 1$. On the other hand, our results show that for $n \asymp \sqrt d$ all tests is powerless for any fixed $\varepsilon \in (0, 2)$. Moreover, our result characterizes the precise regions of impossibility for every level of sparsity $s$, which uncover levels of impossibility based on gradations of the ambient dimensions of the problem. In particular, it is revealed that the minimum sample size required (as a function of $d$) for any possibility of consistent detection decreases as $s$ increases. 
\end{remark}


Finally, it is worth mentioning the recent of results of Donoho and Kipnis \cite{kipnis2019higher,donoho2020two,kipnis2021inability,kipnis2021unification} where the authors provide the precise asymptotics of a HC-type test for a sparse version of the two sample problem for testing the equality of high dimensional discrete distributions. This  series of papers sets the tone and motivations for the asymptotically exact results considered here for the uniformity testing problem. However, since we consider only the uniformity testing problem, in contrast to \cite{kipnis2019higher,donoho2020two,kipnis2021inability,kipnis2021unification} we not only provide sharp analysis of tests but also demonstrate precise information theoretic lower bounds up to to the correct asymptotic multiplicative constant. Moreover, \cite{kipnis2019higher,donoho2020two,kipnis2021inability,kipnis2021unification} only considered a modified version of the problem through a rare-weak  Poisson mixture model, whereas we considered the actual minimax setting with multinomial experiments. This in turn restricts the parameters $\bm p= (p_1, p_2, \ldots, p_d)$ to belong to the $d$-dimensional simplex and allows us to recover sharp impossibility regimes where no signal is detectable asymptotically.   As mentioned before, we believe that this simplex constraint creates the extra phase transition of pathological detection. Finally, it is not too difficult to obtain a part of the rate optimal results for the two sample problem under sparsity using a reduction to the one-sample problem considered here. However, since we focus on precise asymptotics whenever achievable, we leave this for future endeavors.

\section{Proof Outline}\label{sec:outline_pf}

In this section we introduce the relevant test statistics and sketch the key ideas involved in the proofs of Theorems \ref{thm:one_sample_alphaleqhalf} and \ref{thm:one_sample_alphageqhalf}. The first step towards this is to consider the {\it Poisson sampling} scheme where instead of drawing $n$ independent samples  from  a  distribution $\bm p \in \cP([d])$, we first choose $N \sim \dPois(n)$ and then draw $N$ i.i.d samples $X_1, X_2, \ldots, X_N$ from $\bm p = (p_1, p_2, \ldots, p_d)$. This model has the advantage that  
the frequencies $Z_1, Z_2, \ldots, Z_d$, where 
\begin{align}\label{eq:ZN}
Z_j= \sum_{i=1}^N \bm 1\{X_i=j\},
\end{align} are independent with $Z_j \sim \dPois(n p_j)$, for $j \in [d]$. Analogous to \eqref{eq:RT} we define the risk of a test function $T: \bm Z=(Z_1,\ldots,Z_d)\rightarrow \{0,1\}$ in the Poisson sampling scheme as: 
\begin{align}\label{eq:R_poisson}
	\overline{\cR}_{n, d}(T,  s, \varepsilon)= \mathbb{P}_{H_0}\left(T=1\right)+\sup_{\bm{p}\in \mathcal{P}(U([d]),s,\varepsilon)}\mathbb{P}_{\bm{p}}\left(T=0\right) , 
\end{align}	
where $Z_1, Z_2, \ldots, Z_d$ are independent and, for $j \in [d]$, $Z_j \sim \dPois(n/d)$ under $H_0$ and $Z_j \sim \dPois(n p_j)$, for $\bm p \in \mathcal{P}(U([d]),s,\varepsilon )$. Therefore, the minimax risk under the Poisson sampling scheme is:  
\begin{align}\label{eq:R_poisson}
	\overline{\cR}_{n, d}(s, \varepsilon)=\inf_{T}\left\{ \mathbb{P}_{H_0}\left(T=1\right)+\sup_{\bm{p}\in \mathcal{P}(U([d]),s,\varepsilon)}\mathbb{P}_{\bm{p}}\left(T=0\right) \right\}, 
\end{align}	
where the infimum in \eqref{eq:R_poisson} is over all test functions $T: \bm Z=(Z_1,\ldots,Z_d)\rightarrow \{0,1\}$. Using the exponential concentration of  a Poisson distribution around its mean  it can be easily shown that 
$$\cR_{n, d}(s, \varepsilon)=(1+o(1))\overline{\cR}_{n, d}(s, \varepsilon),$$
where $\cR_{n, d}$ is as defined in \eqref{eq:R} (see, for example, \cite[Section A]{wu2016minimax}). This allows us to derive the results in Theorems \ref{thm:one_sample_alphaleqhalf} and \ref{thm:one_sample_alphageqhalf} from the analogous result in the Poisson model. In light of this observation, we will hereafter work in the Poissonized setup.

\subsection{Proof Outline for Theorem \ref{thm:one_sample_alphaleqhalf}}
\label{sec:alphaleqhalfoutline}

Recall that Theorem \ref{thm:one_sample_alphaleqhalf} deals with the dense regime, $0 \leq \alpha\leq \frac{1}{2}$, where the detection boundary depends on whether the sample size $n\gtrsim d^{\frac{1}{2}+\alpha}$ or $n\ll d^{\frac{1}{2}+\alpha}$.

We begin with the case $n\gtrsim d^{\frac{1}{2}+\alpha}$. Here, since the number of samples is `large' enough compared to the dimension and dense perturbations from uniformity are allowed, the problem mimics the classical  (fixed dimension, large sample size) paradigm. Consequently, it is possible to construct a $\chi^2$-type test which attains the minimax detection boundary. In particular, assuming the Poisson sampling model and $Z_1, Z_2, \ldots, Z_d$ as in \eqref{eq:ZN}, we consider the following statistic: 
\begin{align}\label{eq:Tn}
T_n=\sum_{j=1}^d \left\{ \left(Z_j- \frac{n}{d} \right)^2-Z_j \right\}. 
\end{align}
It is easy to check that $T_n/n^2$ is an unbiased estimate of $||\bm p - U([d])||_2^2$ (see Lemma \ref{lm:Tn}). A variance calculation followed by an application of the Chebyshev's  inequality then shows that  the test which rejects  the null hypothesis whenever $T_n$ is `large', is asymptotically powerful for the problem \eqref{eq:uniform_H} whenever $\varepsilon \gg \varepsilon_1(n, d, s)$, where $\varepsilon_1(n, d, s)$ is as defined in \eqref{eq:threshold_epsilon} (see Lemmas \ref{lm:Tn} and \ref{lm:threshold_I} for details).  This establishes the desired lower bound on the signal strength (upper bound on the minimax risk). To show that all tests are powerless below this signal strength we will need to lower bound the minimax risk. 
Towards this, for any test function $T$ and any prior $\pi$ on $\cP(U[d], s, \varepsilon)$ (the alternative space $H_1$), define the Bayes risk of $T$ under the prior $\pi$ as 
 \begin{align*}
	\overline{\cR}_{n, d}(T, s, \varepsilon, \pi)= \mathbb{P}_{H_0}\left(T=1\right)+ \E_{\pi} \left[\mathbb{P}_{\bm{p}}\left(T=0\right) \right],
\end{align*}	
where the expectation is taken over the randomness of $\bm p \sim \pi$. Then defining 
\begin{align}\label{eq:L_pi}
L_{\pi} = \E_{\pi} \left[ \frac{\P_{\bm p}(Z_1, Z_2, \ldots, Z_d)}{\P_{H_0}(Z_1, Z_2, \ldots, Z_d)} \right] = \E_{\pi} \left[ \prod_{j=1}^d e^{- n \left( p_j - \frac{1}{d} \right) } \left(d p_j\right)^{Z_j} \right], 
\end{align} 
as the $\pi$-integrated likelihood ratio, the minimax risk of any test function $T$ can be bounded below as follows: 
\begin{align}\label{eq:R_lb}
\overline{\cR}_{n, d}(T,  s, \varepsilon) \geq \overline{\cR}_{n, d}(T, s, \varepsilon, \pi) 
& \geq 1-\tfrac{1}{2} \E_{H_0}|L_{\pi}-1| \nonumber \\
& \geq 1- \tfrac{1}{2} \sqrt{\E_{H_0}[L_{\pi}^2]-1},
\end{align}
where the last step uses the Cauchy-Schwarz inequality. Therefore, to show that no tests are powerful for $\varepsilon \ll \varepsilon_1(n, d, s)$ it suffices to prove $\E_{H_0}[L_{\pi_n}^2]= 1 + o(1)$, for an appropriately chosen sequence of priors $\pi_n$ on $\cP(U[d], s, \varepsilon)$.  
In this case, the prior $\pi_n$ is constructed as follows: We first decompose the support $[d]$ into $d/2$ consecutive pairs $\{\{1, 2\}, \{3, 4\}, \ldots, \{d-1, d\}\}$, then chose $s/2$ pairs from these $d/2$ pairs uniformly at random, and within each chosen pair we increase a  randomly chosen coordinate by $\varepsilon/s$ and decrease the other by $\varepsilon/s$. In  Lemma \ref{lm:threshold_lb_I} we show that for this prior $\E_{H_0}[L_{\pi_n}^2] = 1 +o(1)$ whenever $\varepsilon \ll \varepsilon_1(n, d, s)$.

Next, we consider the case $n \ll d^{\frac{1}{2}+\alpha}$. In this regime, there are `too few' samples and, as a consequence, all tests are asymptotically powerless for any asymptotically achievable signal strength. To show this we observe from the proof of Proposition \ref{ppn:L01} that the maximum asymptotically achievable signal strength is attained by increasing one coordinate of the domain $[d]$ to $s/d$ and decreasing $s-1$ coordinates to zero. Therefore, a natural least favorable prior in this regime would be to randomly choose $s$ coordinates from the domain uniformly at random, increase the value one of these randomly chosen coordinates to $s/d$ and decrease the values of the remaining $s-1$ coordinates to zero (or making it arbitrarily small so that the likelihood ratio is well-defined). However, the second moment corresponding to this prior is difficult to analyze because of the dependencies introduced by the sampling without replacement scheme. To circumvent this issue, we fix $0 < \delta < 1$ arbitrarily small and first chose a subset $S$ of size $(1-\delta) s$ from the first half of the domain $\{1, 2, \ldots, d/2\}$ uniformly at random. Then depending on the outcome of $(1-\delta) s$ independent coin tosses with success probability $\delta$, we either increase the value of a coordinate in $S$ to $s/d$ or decrease it to arbitrarily close to zero. Finally, to ensure that the resulting vector belongs to $\cP(U[d], s, \varepsilon)$ with high probability, we make necessary adjustments in the second half of the domain $\{ d/2+1, \ldots, d\}$, depending upon the magnitude of the deficit incurred in the first half. The independent sampling in the first half of the domain simplifies the analysis, however, the calculations are still delicate because of the interactions between the first and the second halves of the domain elements in the second moment of the likelihood. Details of the calculation are given in Proposition \ref{ppn:threshold_lb_II} which also includes the impossibility regime in the sparse case (Theorem \ref{thm:one_sample_alphageqhalf} (2)).

\subsection{Proof Outline for Theorem \ref{thm:one_sample_alphageqhalf}}
\label{sec:alphageqhalfoutline}

This theorem deals with the sparse regime, $\frac{1}{2} < \alpha < 1$, where the detection boundary depends on whether the sample size $n\gg d \log d$ or $n\ll d \log ^3 d$. As mentioned before, the regime $n\ll d \log d$, where all tests are powerless irrespective of the signal strength, can be handled by the same arguments as the analogous regime in the dense case (Theorem \ref{thm:one_sample_alphaleqhalf} (2)). Therefore, it suffices to discuss the case $n\gg d \log^3 d$. Here our test relies on a combination of three tests as follows: First, a Bonferroni correction two maximum-type tests, namely those rejecting for large values of 
$$\max_{j \in [d]} Z_j \quad \text{ and } \quad M_{n,d}=\max_{j \in [d]} \left\vert\frac{Z_j - n/d}{\sqrt{n/d}}\right\vert,$$
attains the minimax detection boundary upto large constants. Then to obtain the correct constant in the detection threshold, we need to consider a further Bonferroni correction with a properly calibrated HC statistic. \textcolor{black}{This HC test rejects  for large values of $$\sup_{t\in \mathcal{T}}\frac{\sum_{j=1}^d \{\bm 1\{ |D_j| \geq t \} - \P_{H_0}(|D_j| \geq t)\}}{\sqrt{\sum_{j=1}^d  \P_{H_0}(|D_j| \geq t)(1-\P_{H_0}(|D_j| \geq t))}}$$ for a suitably chosen index set $\mathcal{T}$ (see Lemma \ref{lm:threshold_II} for details).  }

Next, to show all tests are powerless for $\varepsilon \ll \varepsilon_2(n, d, s)$, where $\varepsilon_2(n, d, s)$ is as defined in \eqref{eq:threshold_II}, we slightly modify the sequence of priors from Theorem \ref{thm:one_sample_alphaleqhalf} (2) (as described in Section \ref{sec:alphaleqhalfoutline} above). In particular, for any given $\delta>0$, we first choose $\lceil s/2 \rceil$ elements at random from the first $\lceil d/2 \rceil$ indices and set the corresponding coordinates of $\bm p$ to $\frac{1}{d}+\varepsilon_2(n,d,s) \sqrt{C(\alpha)(1-\delta)}/s$, where $C(\alpha)$ is as defined in \eqref{eq:threshold_II}. This is subsequently balanced out by choosing $s-\lceil s/2 \rceil$ elements at random from the remaining $d-\lceil d/2 \rceil$ indices and setting the corresponding coordinates of $\bm{p}$ to $\frac{1}{d}-\varepsilon_2(n,d,s) \sqrt{C(\alpha)(1-\delta)}/s$. This induces a prior on $\mathcal{P}(U[d],s,\varepsilon)$ in the regime of Theorem Theorem \ref{thm:one_sample_alphageqhalf} (1) and allows us to implement a truncated second moment approach en route deriving the sharp matching constant in the lower bound. 

\section{Numerical Results}
\label{sec:numerical_results}

\begin{figure*}[h]
\centering
\begin{minipage}[l]{0.48\textwidth}
\centering
\vspace{-0.05in}
\includegraphics[width=3.25in]
    {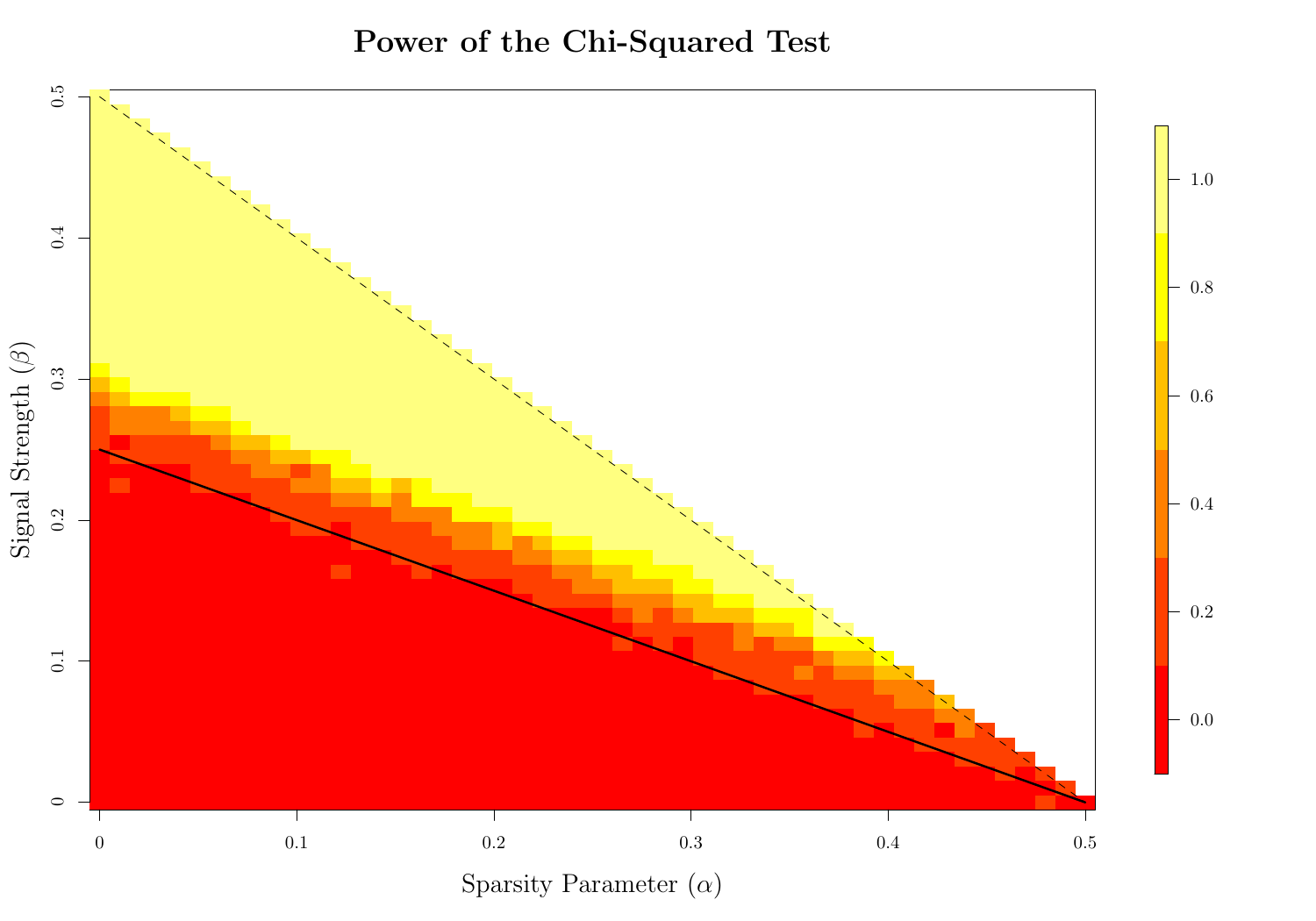}\\
\small{(a)}
\end{minipage} 
\begin{minipage}[l]{0.48\textwidth}
\centering
\vspace{-0.05in}
\includegraphics[width=3.25in]
    {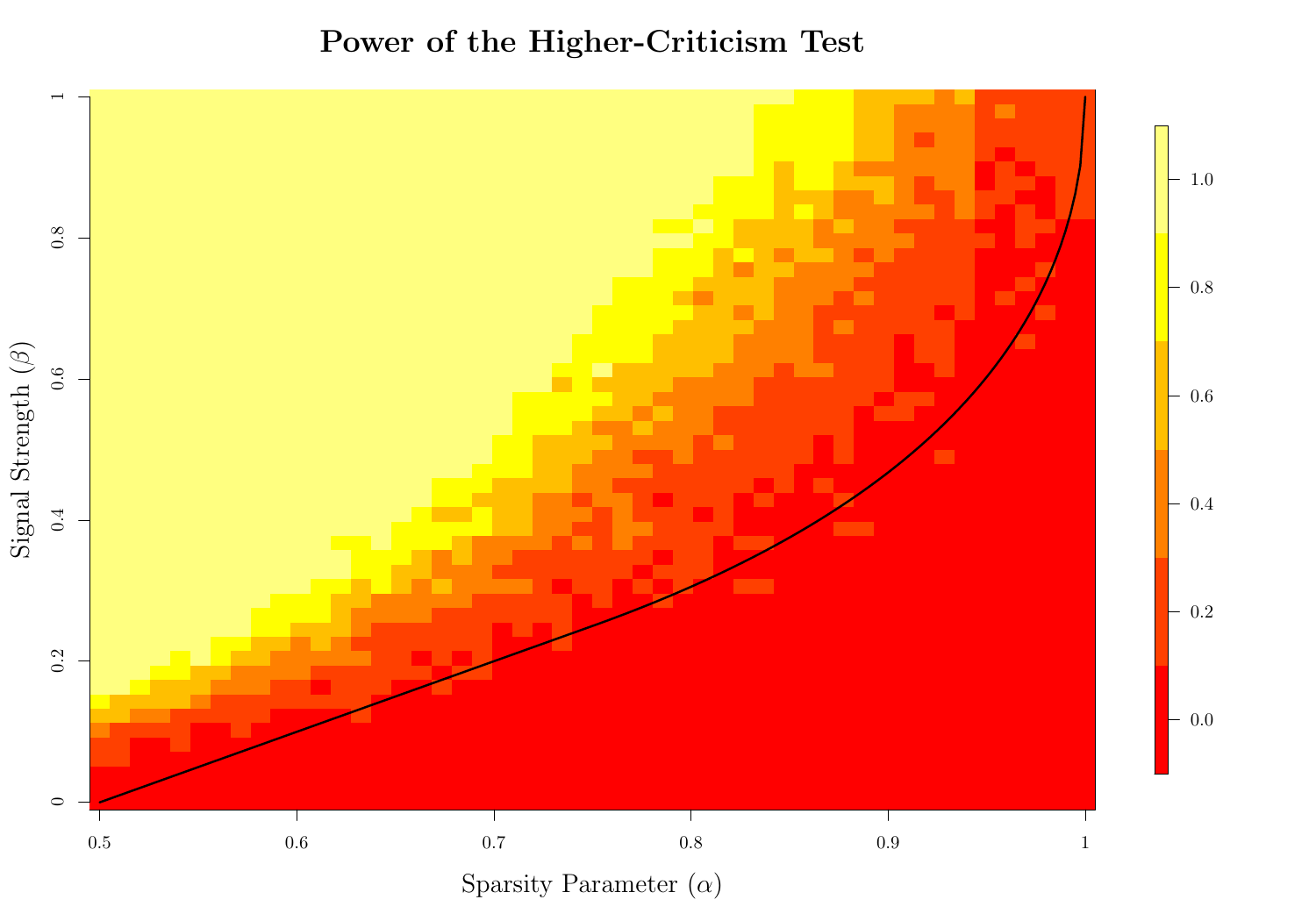}\\
\small{(b)}
\end{minipage} 
\caption{\small{Heatmaps of the empirical powers of (a) the $\chi^2$ test in the dense regime ($\alpha \leq \frac{1}{2}$), and (b) the HC test in the sparse regime ($\alpha > \frac{1}{2}$). The solid black curves denote the theoretical detection thresholds in the respective regimes.}}
\label{fig:T12}
\end{figure*}

In this section we present some numerical experiments in order to demonstrate the behavior of the various tests in finite samples. We begin with the dense regime: $0 \leq \alpha \leq \frac{1}{2}$. Here, we re-parametrize the signal strength as $\varepsilon = d^{\beta}/\sqrt n$, where $\beta \geq 0$. Then the result in Theorem \ref{thm:one_sample_alphaleqhalf} (1) shows that the $\chi^2$ test is asymptotically powerful, whenever 
$$\varepsilon = \frac{d^{\beta}}{\sqrt n} \gg \varepsilon_1(n, d, s) := \left(\frac{s}{n\sqrt{d}} \right)^{\frac{1}{2}},$$ 
which holds if and only if $\beta > \frac{1}{4} -\frac{\alpha}{2}$. On the other hand, all tests are asymptotically powerless whenever $\beta < \frac{1}{4} -\frac{\alpha}{2}$. In other words, the detection boundary of the problem in the dense regime is determined by the line: $\beta=  \frac{1}{4} -\frac{\alpha}{2}$. In order to illustrate this phenomenon numerically, we choose $d=n=10^5$ and compute, for a fixed value of $(\alpha, \beta)$, the empirical power (over 100 repetitions) of $\chi^2$ test (based on the statistic $T_n$ defined in \eqref{eq:Tn}), against the alternative distribution  $\bm p = (p_1, p_2, \ldots, p_d)$, where 
\begin{align}\label{eq:p_dense}
p_j= 
		\left\{
		\begin{array}{cl}
		\frac{1}{d} + \frac{\varepsilon}{s}   &  \text{for } 1 \leq j \leq \frac{s}{2}, \\ 
		\frac{1}{d} - \frac{\varepsilon}{s}   &  \text{for } \frac{s}{2}+1 \leq j \leq s, \\ 
		\frac{1}{d} &  \text{for } s+1 \leq j \leq d . 
		\end{array}
		\right. 
\end{align}		
The cutoff of the test is obtained by computing the 95\% empirical quantile of $T_n$ based on 1000 simulation instances of the null distribution (which is the uniform distribution on $[d]$). Figure \ref{fig:T12} (a) shows the heatmap of the empirical power as $(\alpha, \beta) $ varies over a $50 \times 50$ grid in $[\frac{1}{2}, 1] \times [0, \frac{1}{2}]$. Note that the upper half of the grid is not included in the plot because $\bm p \in \cP(U([d]), s, \varepsilon)$ if and only if $\alpha + \beta \leq \frac{1}{2}$.\footnote{Clearly, for $\bm p$ as in \eqref{eq:p_dense}, $\sum_{j=1}^d p_j=1$,  $||\bm p - U([d])||_0 =s$, and $||\bm p - U([d])||_1 \geq \varepsilon$. Hence, $\bm p \in \cP(U([d]), s, \varepsilon)$  as long as $p_j \geq 0$, for all $j \in [d]$. This happens whenever $1/d \geq \varepsilon/s$, that is, if and only if, $\alpha + \beta \leq \frac{1}{2}$, since $n=d$ in this example.}  The solid black line denotes the theoretical detection boundary, that is, the curve $\beta=  \frac{1}{4} -\frac{\alpha}{2}$. As predicted by our theoretical results, the phase transition of the empirical power of the $\chi^2$ test (between regions where it has maximal power and the region where it has no power) is clear from the simulations as well. 


Next, we consider the sparse regime: $\frac{1}{2} < \alpha < 1$. Here, we re-parametrize the signal strength as 
$$ \varepsilon  = s \left(\frac{2 \beta \log{d}}{nd} \right)^{\frac{1}{2}},$$
where $\beta \geq 0$. 
Then the result in Theorem \ref{thm:one_sample_alphageqhalf} (1) implies that detection is possible if only if $\beta > C(\alpha)$ (where $C(\alpha)$ is as defined in \eqref{eq:threshold_III}). To illustrate this phenomenon numerically, we choose $d=10^5$, $n= d^{1.4}$, set the cutoff at 90\%, and compute, for a fixed value of $(\alpha, \beta)$, the empirical power (over 100 repetitions) of the HC test 
against the alternative distribution $\bm p = (p_1, p_2, \ldots, p_d)$ as in \eqref{eq:p_dense}. 
Figure \ref{fig:T12} (b) shows the heatmap of the empirical power as $(\alpha, \beta) $ varies over a $50 \times 50$ grid in $[\frac{1}{2}, 1] \times [0, 1]$. The solid black curve denotes the theoretical detection boundary, that is, the curve $\beta=  C(\alpha)$. As before, the phase transition of the empirical power of the HC test is clear from the simulations, validating the theoretical predictions.


\section{Proofs of the Main Results}\label{sec:main_results_pf}
		
	
This section is organized as follows: In Section \ref{sec:leqhalfpf} we prove  Theorem \ref{thm:one_sample_alphaleqhalf} (1). The proof of Theorem \ref{thm:one_sample_alphageqhalf} (2) is given in Section \ref{sec:geqhalfpf}. The proof for the impossibility regimes (Theorem \ref{thm:one_sample_alphaleqhalf} (2) and \ref{thm:one_sample_alphageqhalf} (2)) is given in Section \ref{sec:lower_bound_pf}. 
	
\subsection{Proof of Theorem \ref{thm:one_sample_alphaleqhalf} (1)}
\label{sec:leqhalfpf}
	
In Section \ref{sec:leqhalf_upper_bound_pf} we prove the upper bound on the minimax risk (lower bound on the detection boundary) by analyzing performance of the $\chi^2$-type statistic \eqref{eq:Tn}. The matching lower bound on the minimax risk is established in Section \ref{sec:leqhalf_lower_bound_pf}.	
	
\subsubsection{Upper Bound for Theorem {\em \ref{thm:one_sample_alphageqhalf} (1)}: Analysis of the Chi-Squared Test}
\label{sec:leqhalf_upper_bound_pf}

Throughout we suppose $N \sim \dPois(n)$ and $Z_1, Z_2, \ldots, Z_d$ are as defined in \eqref{eq:ZN}.  Recall the definition of the $\chi^2$-type statistic $T_n$ from \eqref{eq:Tn}. 
	
	\begin{lem}\label{lm:Tn} For $T_n$ as defined in \eqref{eq:Tn}, the following hold: 
		\begin{enumerate}
			
			\item[(a)] For $\bm p = (p_1, p_2, \ldots, p_d) \in \cP([d])$, 
			\begin{align}\label{eq:ETn}
			\E_{\bm p}[T_n]=n^2 ||\bm p - U([d])||_2^2. 
			\end{align}
			This implies, for $\bm p \in \cP(U[d], s, \varepsilon)$, $\E_{\bm p}[T_n] \geq \frac{n^2 \varepsilon^2}{2s}.$
			
			\item[(b)] For $\bm p = (p_1, p_2, \ldots, p_d) \in \cP([d])$, 
				\begin{align}\label{eq:VTn}
				\Var_{\bm p}[T_n]=\sum_{j=1}^d \Big\{ \frac{2n^2}{d^2} + 2n^2 \Delta_j^2 + \frac{4n^3 \Delta_j^2}{d} - 4n^3 \Delta_j^3 \Big\}, 
				\end{align}
				where $\Delta_i=\frac{1}{d}-p_i$, for $1 \leq i \leq d$.  
		\end{enumerate} 
	\end{lem}
	
	\begin{proof} The result in \eqref{eq:ETn} follows easily from the fact that  $Z_1, Z_2, \ldots, Z_d$ are independent with $Z_j \sim \dPois(n p_j)$, for $j \in [d]$. Hence, for $\bm p \in \cP(U[d], s, \varepsilon)$, by the Cauchy Schwarz inequality, 
		$$\E_{\bm p}[T_n] = n^2 ||\bm p - U([d]) ||_2^2  \geq \frac{n^2}{s} ||\bm p - U([d]) ||_1^2  \geq \frac{n^2 \varepsilon^2}{2s},$$
		since only $s$ coordinates of the vector $(\Delta_1, \Delta_2, \ldots, \Delta_d)$ are non-zero. 
		
		The result in \eqref{eq:VTn} follows from \cite[Appendix B]{balakrishnan2019hypothesis}. 
	\end{proof}
	
The following lemma shows that the test which rejects for large values of $T_n$ attains the minimax detection boundary in the regime $\alpha \leq \frac{1}{2}$ and $n \gtrsim d^{\frac{1}{2}+\alpha}$. This establishes part (a) of Theorem \ref{thm:one_sample_alphaleqhalf} (1).

	\begin{lem}\label{lm:threshold_I} Suppose $\alpha \leq \frac{1}{2}$ and $n \gtrsim d^{\frac{1}{2}+\alpha}$. Then the test function 
		\begin{align}\label{eq:Tn_I}
		\phi=\bm 1\left\{ T_n \geq \frac{n^2 \varepsilon^2}{4s} \right\}, 
		\end{align}
is asymptotically powerful whenever $\varepsilon \gg \varepsilon_1(n, d, s)$, where $\varepsilon_1(n, d, s)$ is as defined in \eqref{eq:threshold_epsilon}. 
	\end{lem}
	
	\begin{proof} Note that by Lemma \ref{lm:Tn}, $\E_{H_0}[T_n] = 0$ and $\Var_{H_0}[T_n]= \frac{2n^2}{d}.$ Then by the Chebyshev's inequality,  the probability of the Type I error can be bounded by 
		$$\P_{H_0}(\phi=1)=\P_{H_0}\left(T_n \geq \frac{n^2 \varepsilon^2}{4s}\right) \lesssim \frac{s^2\Var_{H_0}[T_n]}{n^4 \varepsilon^4} \lesssim \frac{s^2}{d n^2 \varepsilon^4} \rightarrow 0,$$ 
		whenever $\varepsilon \gg \varepsilon_1(n, d, s)$. 

		Next, we consider the probability of Type II error. For this, suppose $\bm p \in \cP(U([d]), s, \varepsilon)$. Then using Lemma \ref{lm:Tn} gives, 
		\begin{align} 
		\Var_{\bm p}[T_n] & = \sum_{j=1}^d \left( \frac{2 n^2}{d^2} + 2 n^2 \Delta_j^2 + \frac{4n^3 \Delta_j^2}{d} - 4 n^3 \Delta_j^3 \right) \nonumber \\ 
		& \leq  \frac{2n^2}{d}+  \sum_{j=1}^d n^2 \Delta_j^2 + \frac{4 n^3}{d}\sum_{j=1}^d  \Delta_j^2 + 4 \left(\sum_{j=1}^d  n^2 \Delta_j^2 \right)^{\frac{3}{2}} \tag*{(by H\"older's inequality)} \nonumber \\ 
		& \leq  \frac{2n^2}{d} + \E_{\bm p}[T_n] + \frac{4 n}{d} \E_{\bm p}[T_n] + 4 \left( \E_{\bm p}[T_n]  \right)^{\frac{3}{2}} \tag*{(since $\E_{\bm p}[T_n] = n^2  \sum_{j=1}^d \Delta_j^2$)} \nonumber \\ 
		&:= U_1+U_2+U_3+U_4.
	\label{eq:var_Tn}	
		\end{align} 
Note that by Lemma \ref{lm:Tn} and using $\varepsilon \gg \varepsilon_1(n, d, s)$ and $n \gtrsim d^{\frac{1}{2}+\alpha}$, 
		$$\E_{\bm p}[T_n]= n^2 \sum_{j=1}^d  \Delta_j^2 \geq \frac{n^2 \varepsilon^2}{4s} \gg \frac{n}{\sqrt d} \gtrsim d^{\alpha}  \gtrsim 1.$$
		This implies, $U_2=o(\E_{\bm p}[T_n]^2)$ and $U_4=o(\E_{\bm p}[T_n]^2)$. Next, consider the first term $U_1:=\frac{2n^2}{d}$. For this, using $\E_{\bm p}[T_n] \geq \frac{n^2 \varepsilon^2}{4s}$ gives,   
		$$\frac{U_1}{(\E_{\bm p}[T_n])^2} \lesssim \frac{s^2}{d n^2 \varepsilon^4} \rightarrow 0,$$
		whenever $\varepsilon \gg \varepsilon_1(n, d, s)$. Similarly, it follows that $U_3=o(\E_{\bm p}[T_n]^2)$, whenever $\varepsilon \gg \varepsilon_1(n, d, s)$.  Therefore, by \eqref{eq:var_Tn}, $\frac{\Var_{\bm p}[T_n] }{(\E_{\bm p}[T_n])^2} \rightarrow 0$, whenever $\bm p \in \cP(U([d]), s, \varepsilon)$ and $\varepsilon \gg \varepsilon_1(n, d, s)$. Therefore, the probability of Type II error of the test $\phi$ is 
		\begin{align*} 
		\P_{\bm p}(\phi=0)=\P_{\bm p}\left(T_n \leq \frac{n^2 \varepsilon^2}{4s}\right) & \leq \P_{\bm p}\left(T_n -\E_{\bm p}[T_n] \leq - \tfrac{1}{2} \E_{\bm p}[T_n] \right) \lesssim \frac{\Var_{\bm p}[T_n] }{(\E_{\bm p}[T_n])^2} \rightarrow 0, 
		\end{align*} 
		whenever $\varepsilon \gg \varepsilon_1(n, d, s)$. 
	\end{proof}
	
	\begin{remark} Note that if $n \gg d^{\frac{1}{2}+\alpha}$, then 
		$$\varepsilon_1(n, d, s) = \left(\frac{s}{n\sqrt{d}}\right)^{\frac{1}{2}}=\left(\frac{d^{\frac{1}{2}-\alpha}}{n}\right)^{\frac{1}{2}} \ll  \frac{s}{d} \lesssim \varepsilon_{\max} ,$$
where $\varepsilon_{\max}$ is as defined in Proposition \ref{ppn:L01}. This means in the regime $n\gg d^{\frac{1}{2}+\alpha}$, there is always a non-trivial range of detectable signal strength $\varepsilon_1(n, d, s) \ll \varepsilon \leq \varepsilon_{max}$. 
	\end{remark}

		\subsubsection{Lower Bound for Theorem {\em \ref{thm:one_sample_alphaleqhalf} (1)}} 
\label{sec:leqhalf_lower_bound_pf}

	In this section we prove part (b) of Theorem \ref{thm:one_sample_alphaleqhalf} (1). In other words, we show that the detection threshold of the $\chi^2$ test derived in Lemma \ref{lm:threshold_I} above is tight in the regime $\alpha \leq \frac{1}{2}$ and $n \gtrsim d^{\frac{1}{2}+\alpha}$. This is summarized in the following lemma:  
	
	\begin{lem}\label{lm:threshold_lb_I} Suppose $\alpha \leq \frac{1}{2}$ and $n \gtrsim d^{\frac{1}{2}+\alpha}$. Then all tests are asymptotically powerless whenever $\varepsilon \ll \varepsilon_1(n, d, s)$.
	\end{lem}

	\begin{proof} 
		Throughout we assume that $d$ and $s=2K$ are even. Decompose the domain $[d]$ into $\frac{d}{2}$ consecutive pairs $$\{\{1, 2\}, \{3, 4\}, \ldots, \{d-1, d\}\}.$$ 
Next, suppose  $\bm \eta=(\eta_1, \eta_2, \ldots, \eta_{d/2})$ are i.i.d. random variables taking values $\{\pm 1\}$ with probability $\frac{1}{2}$. Now, choose a subset $S$ of $[\frac{d}{2}]$ of size $|S|=K$, uniformly at random. Denote the elements of $S$ by $\{i_1, i_2, \ldots, i_K \}$.  Then define, 
		$$p_{2i_r-1}= \frac{1}{d} +  \eta_r \frac{\varepsilon}{s}   \quad \text{ and }  \quad p_{2i_r}= \frac{1}{d}  - \eta_r \frac{\varepsilon}{s},
		$$ 
		for $r \in [K]$ and $p_j=\frac{1}{d}$, otherwise. (In other words, for each tuple in $S$ we increase the value in one of the two coordinates by $\varepsilon/s$ and decrease the value in the other coordinate by $\varepsilon/s$, uniformly at random.) Clearly, this ensures $\bm p \in \mathcal P(U[d], s, \varepsilon)$. 
		
		Denoting the prior above on $\cP(U([d]), s, \varepsilon)$ by $\pi_n$, the $\pi_n$-integrated likelihood ratio (recall \eqref{eq:L_pi}) becomes, 
		\begin{align}
		L_{\pi_n} 
		& =  \E_{S, \bm \eta } \left[ \prod_{r \in S} e^{-\frac{n \eta_r \varepsilon}{s}} \left( 1+ \frac{\eta_r d \varepsilon}{s} \right)^{Z_{2i_r-1}}  e^{\frac{n \eta_r  \varepsilon}{s}} \left( 1- \frac{\eta_r d \varepsilon}{s} \right)^{Z_{2i_r}}  \right] 
		\nonumber \\
		& =  \E_{S, \bm \eta } \left[ \prod_{r \in S} e^{-\frac{n \eta_r \varepsilon}{s}} \left( 1+  \eta_r \theta \right)^{Z_{2i_r-1}}  e^{\frac{n \eta_r \varepsilon}{s}} \left( 1- \eta_r \theta \right)^{Z_{2 i_r}}   \right], \nonumber  
		\end{align}
		where $\E_{S, \bm \eta }$ denotes the expectation with respect to the randomness of the set $S$ and $\bm \eta= (\eta_1, \eta_2, \ldots, \eta_K)$, and $\theta=\frac{d \varepsilon}{s}$. This implies, 
		\begin{align}
		& \E_{H_0} [L_{\pi_n}^2] \nonumber \\ 
		&= \E_{\substack{{S_1, \bm \eta}\\ {S_2, \bm \eta'}}}\E_{H_0}\left[  \prod_{r \in  S_1 \cap S_2} e^{-\frac{n (\eta_r +\eta_r') \varepsilon}{s}} \left( (1+  \eta_r \theta) (1+ \eta_r' \theta) \right)^{Z_{2i_r-1}}  e^{\frac{n (\eta_r +\eta_r') \varepsilon}{s}}  \left( (1-  \eta_r \theta) (1- \eta_r' \theta) \right)^{Z_{2i_r}}  \right], \nonumber 
		\end{align}
		since 
		$$\E_{H_0}\left[  \prod_{r \in  S_1 \triangle S_2} e^{-\frac{n (\eta_r +\eta_r') \varepsilon}{s}} \left( (1+  \eta_r \theta) (1+ \eta_r' \theta) \right)^{Z_{2i_r-1}}  e^{\frac{n (\eta_r +\eta_r') \varepsilon}{s}}  \left( (1-  \eta_r \theta) (1- \eta_r' \theta) \right)^{Z_{2i_r}}  \right] =1,$$
		where $S_1 \triangle S_2$ denotes the symmetric difference of the sets $S_1$ and $S_2$. Therefore, 
		\begin{align}\label{eq:L2_I}
		\E_{H_0} [L_{\pi_n}^2] & = \E_{\substack{{S_1, \bm \eta}\\ {S_2, \bm \eta'}}}\E_{H_0}\left[  \prod_{r \in  S_1 \cap S_2} e^{ \frac{n}{d}( (\eta_r +\eta_r')\theta+ \eta_r \eta_r' \theta^2)}  e^{ - \frac{n}{d}( (\eta_r +\eta_r')\theta - \eta_r \eta_r' \theta^2)}  \right] \nonumber \\ 
		&= \E_{\substack{{S_1, \bm \eta}\\ {S_2, \bm \eta'}}}\E_{H_0}\left[  \prod_{r \in  S_1 \cap S_2} e^{2 \eta_r \eta_r' \frac{n \theta^2}{d}}  \right],
		\end{align}
		where the second last step uses, 
		\begin{align}\label{eq:exptheta}
		\E_{H_0}  \left( (1+  \eta_r \theta) (1+ \eta_r' \theta) \right)^{Z} = e^{\frac{n}{d}( (\eta_r +\eta_r')\theta+ \eta_r \eta_r' \theta^2)}, 
		\end{align} 
		for $Z \sim \dPois(n/d)$. Now, let $(S_1\cap S_2)^+=\{r \in S_1 \cap S_2: \eta_r=\eta_r'\}$ and $(S_1\cap S_2)^-=\{r \in S_1 \cap S_2: \eta_r \ne \eta_r'\}$. Then \eqref{eq:L2_I} gives, 
		\begin{align}\label{eq:L2_II}
		\E_{H_0} & [L_{\pi_n}^2] = \E_{\substack{{S_1, \bm \eta}\\ {S_2, \bm \eta'}}}\E_{H_0}\left[ \exp\left\{\frac{2n  \varepsilon^2 d}{s^2} |(S_1\cap S_2)^+|  -\frac{2n  \varepsilon^2 d}{s^2} |(S_1\cap S_2)^-| \right\}   \right] \nonumber \\ 
		&=\E_{\substack{{S_1, \bm \eta}\\ {S_2, \bm \eta'}}} \left[ \exp\left\{\frac{4 n d \varepsilon^2}{s^2 } \left(|(S_1 \cap  S_2)^+| \right) -\frac{2n d \varepsilon^2}{s^2 } \left(|S_1  \cap  S_2)|  \right)\right\} \right]. 
		\end{align}
Note that $|(S_1 \cap  S_2)^+| \Big| S_1  \cap  S_2$ is distributed as $\dBin(|S_1  \cap  S_2|, \frac{1}{2})$. Then, \eqref{eq:L2_II} gives, 
		\begin{align*}
		\E_{H_0}[L_{\pi_n}^2] = \E_{S_1 \cap S_2} \left[e^{ -\frac{2 n d \varepsilon^2}{s^2 } |S_1  \cap  S_2| }  \left[ \frac{1}{2}+ \frac{1}{2} e^{\frac{4 n d \varepsilon^2}{s^2 }} \right]^{|S_1  \cap  S_2| } \right]  =  \left[  \cosh\left(\frac{2 n d \varepsilon^2}{s^2 }\right) \right]^{|S_1  \cap  S_2| }. 
		\end{align*}
		Now, using the fact that $|S_1 \cap S_2| \sim \mathrm{Hypergeometric}(d/2, s/2, s/2)$, which is dominated by the $\dBin(s/2, s/d)$ distribution in convex ordering \cite[Proposition 20.6]{aldous1985exchangeability} it follows that,  
		\begin{align}\label{eq:L_H0}
		\E_{H_0} [L_{\pi_n}^2]  \leq  \left[ 1-\frac{s}{d}  +\frac{s}{d} \cosh\left(\frac{2n d \varepsilon^2}{s^2 }\right) \right]^{\frac{s}{2}} & = \exp\left\{\frac{s}{2} \log\left( 1-\frac{s}{d}  +\frac{s}{d} \cosh\left(\frac{2n d \varepsilon^2}{s^2 }\right) \right)\right\}\nonumber \\  
		& \leq \exp\left\{\frac{s^2}{2 d} \left( \cosh\left(\frac{2n d \varepsilon^2}{s^2 }\right) -1 \right)\right\}. 
		\end{align} 
		Note that the assumption $\varepsilon \ll \varepsilon_1(n, d, s)$ implies that $\frac{n d \varepsilon^2}{s^2 } \ll 1$, since $\frac{n d \varepsilon^2}{s^2 } \ll \frac{\sqrt d}{s}\ll 1$, since $\alpha \leq \frac{1}{2}$. Then using $\cosh x -1 = (1+o(1)) x^2$, as $x \rightarrow 0$, from \eqref{eq:L_H0} we get, 
		$$\E_{H_0} [L_{\pi_n}^2] \leq \exp\left\{\frac{s^2}{2d} \left( \cosh\left(\frac{2 n d \varepsilon^2}{s^2 }\right) -1 \right)\right\} = \exp\left\{ (1+o(1)) \frac{2 n^2 d \varepsilon^4}{s^2} \right\} =1+o(1),$$
		because $\frac{n^2 d \varepsilon^4}{s^2} \ll 1$, when $\varepsilon \ll \varepsilon_1(n, d, s)$. 
		Therefore, by \eqref{eq:R_lb}, for any test function $T$, $\overline{\cR}_{n, d}(T, s, \varepsilon)  \rightarrow 1$, for $\varepsilon \ll \varepsilon_1(n, d, s)$. This completes the proof of Proposition \ref{ppn:threshold_lb_II}. 
	\end{proof}

	\subsection{Proof of Theorem \ref{thm:one_sample_alphageqhalf} (1)}
\label{sec:geqhalfpf}
	
	In Section \ref{sec:geqhalf_upper_bound_pf} we prove the upper bound on the minimax risk by analyzing performances of maximum-type tests and the HC test. The matching lower bound on the minimax risk is established in Section \ref{sec:geqhalf_lower_bound_pf}.

	\subsubsection{Upper Bound for Theorem {\em \ref{thm:one_sample_alphageqhalf} (1)}}
	\label{sec:geqhalf_upper_bound_pf}
	
As before, suppose $N \sim \dPois(n)$ and $Z_1, Z_2, \ldots, Z_d$ are as defined in \eqref{eq:ZN}. Define,  
\begin{align}\label{eq:Mnd}
M_{n,d}=\max_{j \in [d]} \left\vert\frac{Z_j - n/d}{\sqrt{n/d}}\right\vert. 
\end{align}
We begin by showing that the test which rejects for large values of $M_{n,d}$ or large values $\max_{j \in [d]} Z_j $ attains optimal rate of detection in the regime $\alpha > \frac{1}{2}$ and $n \gg d \log d$.  
	
	\begin{lem}\label{lm:threshold_I} Suppose $\alpha >  \frac{1}{2}$ and $n \gg d \log d$. Then for any $C_0>1$ there exists $C > 1$ such that the test function 
		\begin{align}\label{eq:phi_12}
		\phi= \bm 1 \left\{ \{ M_{n,d}> \sqrt{2 C \log d} \} \bigcup \left\{\max_{j \in [d]} Z_j >\frac{10 n}{d}  \right\} \right\} , 
		\end{align} 
		is (uniformly) asymptotically powerful against any $\bm p\in \mathcal{P}(U([d]),s,\varepsilon)$ whenever there exists a $j\in [d]$ such that $|\Delta_j|=|\frac{1}{d}-p_j|\geq \sqrt{2C_0\log{d}/(nd)}$. 
	\end{lem}

	\begin{proof} For a cutoff $t_{n,d}$ denote by $T_1(t_{n,d})=\mathbf{1}\left\{M_{n,d}>t_{n,d}\right\}$. Then, under $H_0$, by a union bound followed and by a moderate deviation bound for the Poisson distribution we have the following: 
		\begin{align*}
		\mathbb{P}_{H_0}\left(T_1(t_{n,d})=1\right)&\leq d   \mathbb{P}_0\left(\left\vert\frac{Z_1-n/d}{\sqrt{n/d}}\right\vert>t_{n,d}\right)\lesssim d\exp\left\{-\frac{t_{n,d}^2}{2}(1+o(1))\right\},
		\end{align*} 
		provided $t_{n,d}\rightarrow \infty$ is such that $t_{n,d}^2/(n/d)\rightarrow 0$ (see \cite[Lemma 2]{arias2015sparse}). Note that $t_{n,d}  : = \sqrt{2C\log{d}}$ satisfies this condition since $n\gg d\log{d}$, and thereby for any $C>1$, 
		\begin{align}\label{eq:T1_uniform}
		\mathbb{P}_{H_0}\left(T(\sqrt{2C \log d})=1\right) \rightarrow 0.
		\end{align}  
		Now, define $M_{n,d}'=\max_{j \in [d]} Z_j$ and consider test function
		$$T_2(t_{n,d}')=\mathbf{1}\left\{M_{n,d}'>\frac{n}{d}(1+t_{n,d}')\right\},$$
		for a cutoff $t_{n, d}'$. Then by Chernoff bound for $t_{n,d}':=9 \geq e^2$ the following holds (with $h(x)=x\log{x}-x+1\geq \frac{1}{2} x\log{x}$ for $x\geq e^2$):
		\begin{align}\label{eq:T2_uniform}
		\P_{H_0}(T_2(t_{n,d}')=1)\leq d \exp\left\{-\frac{n}{d}h(1+t_{n,d}')\right\}\leq \exp\left\{-\frac{n}{2d}t_{n,d}'\log{t_{n,d}'}-\log{d}\right\}\rightarrow 0,
		\end{align}
		since $n\gg d\log{d}$. Therefore, recalling the definition of $\phi$ from \eqref{eq:phi_12} and combining \eqref{eq:T1_uniform} and \eqref{eq:T2_uniform}, gives 
		$$\P_{H_0}(\phi =1) \leq \mathbb{P}_{H_0}\left(T_1(t_{n,d})=1\right) + \mathbb{P}_{H_0}\left(T_2(t_{n,d}')=1\right) \rightarrow 1,$$
		which shows that the probability of the Type I error goes to zero. 
		
		For the analysis of the Type II consider $\bm p \in \cP(U([d]), s, \varepsilon)$ 
		and let ${j}\in [d]$ be such that $|\Delta_{{j}}|\geq \sqrt{2C_0 \log{d}/(nd) }$ for $C_0>1$ a fixed constant. Now, we consider the following two cases (recall $\Delta_{{j}}= \frac{1}{d} - p_{{j}}$): 
		
		\begin{itemize} 
			
			\item Suppose $d|\Delta_{{j}}|\ll \log{d}$. Then \begin{align*}
			\frac{\pm t_{n,d}\sqrt{n/d}-n\Delta_{{j}}}{\sqrt{np_{j}}}\leq \frac{\pm \sqrt{\frac{2Cn\log{d}}{d}}-\sqrt{\frac{C_0n\log{d}}{d}}}{\sqrt{np_{{j'}}}}
			\end{align*}
			Consequently, any fixed $C>0$ satisfying $2C_0>2C$, 
			\begin{align}\label{eq:T1_p}
			\P_{\bm p}\left(T(t_{n,d}')=0\right)&\leq \P_{\bm p}\left(\frac{ -t_{n, d}\sqrt{n/d}-n\Delta_{{j}}}{\sqrt{np_{j}}}\leq \frac{Z_{{j}}-np_{{j}}}{\sqrt{np_{{j}}}}\leq \frac{ t_{n, d}\sqrt{n/d}-n\Delta_{{j}}}{\sqrt{np_{j}}}\right) \nonumber \\
			&\leq \P_{\bm p}\left(\left\vert\frac{Z_{{j}}-np_{{j}}}{\sqrt{np_{{j}}}}\right\vert>\left(\sqrt{2C_0}-\sqrt{2C}\right)\sqrt{\frac{\log{d}}{1- d\Delta_{{j'}}}}\right) \nonumber \\ 
			& \lesssim \frac{1 - d |\Delta_j|}{\log d} \rightarrow 0, 
			\end{align}
			by Chebyshev's inequality, because $d|\Delta_{{j}}|\ll \log{d}$. 
			
			\item Next, suppose $d|\Delta_{j}|\gtrsim \log{d}$.  This implies that for sufficiently large $d$, $p_{{j}}\geq \frac{ C' \log{d}}{d}$ for some constant $C'>0$. Therefore, 
			\begin{align}\label{eq:T2_p}
			\P_{\bm p}\left(T_2(t_{n,d}')=0\right)\leq \P_{\bm p}\left(Z_{j}\leq \frac{n}{d}(1+t_{n,d}')\right)  &=\P\left(\dPois(np_{{j}})\leq  \frac{10 n}{d}\right) \nonumber \\
			&\leq \P\left(\dPois\left(\frac{C' n \log{d}}{d}\right)\leq  \frac{10 n}{d}\right) \nonumber \\
			& \rightarrow 0,
			\end{align}
			by Markov's inequality. 
		\end{itemize} 
		Therefore, for $\bm p \in \cP(U([d]), s, \varepsilon)$ satisfying the stipulation of the lemma and $\phi$ as in \eqref{eq:phi_12}, combining \eqref{eq:T1_p} and \eqref{eq:T2_p}, gives 
		$$\P_{\bm p}(\phi =0) \leq \mathbb{P}_{\bm p}\left(T_1(t_{n,d})=0\right) + \mathbb{P}_{\bm p}\left(T_2(t_{n,d}')=0\right) \rightarrow 0,$$
		which shows that the probability of Type II error of the test function $\phi$ goes to zero. 
	\end{proof} 
	
Note that the test in Lemma \ref{lm:threshold_I} suffices for consistent detection if for some $\delta>0$ at least one of the coordinates of the alternative $\bm p$ has a deviation larger than $\sqrt{2(1+\delta)\log{d}/(nd)}$ from the null $U([d])$ and $n\gg d\log{d}$. To match the leading constant as in \eqref{eq:threshold_II}, we need to combine the test in \eqref{eq:phi_12} with the HC test. However, for the analysis of the HC test to be introduced next we will need to assume $$\max_{j\in [d]} |\Delta|_j \leq \sqrt{\frac{2 \log{d}}{nd}}.$$ This is enough for sufficiently large $n,d$ since otherwise the test in \eqref{eq:phi_12} detects the corresponding alternative. To this end, let 
\begin{align}\label{eq:D}
D_j:= \frac{Z_j -n/d}{\sqrt{n/d}},
\end{align} 
for $j \in [d]$. Then for $t \in \R$, define the HC statistic as: 
\begin{align}\label{eq:ghct}
GHC(t) := \frac{HC(t)}{\sqrt{\Var(HC(t))}} = \frac{\sum_{j=1}^d \{\bm 1\{ |D_j| \geq t \} - \P_{H_0}(|D_j| \geq t)\}}{\sqrt{\sum_{j=1}^d  \P_{H_0}(|D_j| \geq t)(1-\P_{H_0}(|D_j| \geq t))}},
\end{align}
where 
$$HC(t) := \sum_{j=1}^d \{\bm 1\{ |D_j| \geq t \} - \P_{H_0}(|D_j| \geq t)\}.$$  
Hereafter, given a signal strength $\varepsilon > 0$ we will chose the threshold $t$ in \eqref{eq:ghct} as 
	$$t_r: = \sqrt{2 r \log d},$$
where $r=\min\{1, 4 C^*\}$, with $C^*:=\varepsilon/\varepsilon_2(n, d, s)$ (recall the definition of $\varepsilon_2(n, d, s)$ from \eqref{eq:threshold_II}).  The following result shows that the HC test attains the minimax detection threshold when $\alpha > \frac{1}{2}$ and $n \gg d \log d$, whenever  $d \max_{1 \leq j \leq d}|\Delta_j| \leq \sqrt{2 d\log d/n}$ (recall $\Delta_j:=\frac{1}{d}-p_j$, for $j \in [d]$). (As discussed above, note that if $d \max_{1 \leq j \leq d}|\Delta_j| \gtrsim \sqrt{d\log d/n}$, then the proof of Lemma \ref{lm:threshold_I} shows that the max test can be used.) 
	
	\begin{lem}\label{lm:threshold_II} Suppose $\alpha > \frac{1}{2}$ and $n \gg d \log^3 d$. Then the test that rejects for $$\max\limits_{t\in \left \{\sqrt{2 L \log{d}}  : ~ L \in (0,5) \right \} \bigcap \mathbb{N}} |GHC(t)| > \log d$$ is  asymptotically powerful 
	whenever $\varepsilon \gg \varepsilon_2(n, d, s)$, where  $\varepsilon_2(n, d, s)$ is as defined in \eqref{eq:threshold_II} and $\bm p \in \cP(U[d], s, \varepsilon)$, is such that $d \max_{1 \leq j \leq d}|\Delta_j| \leq \sqrt{2 d\log d/n}$. 
	\end{lem} \medskip
	
\noindent \textit{Proof of Lemma} \ref{lm:threshold_II}: Note that for any $t\in \{\sqrt{2L\log{d}}: L\in (0,5)\}\cap \mathbb{N}$, $\E_{H_0}[GHC(t)]=0$ and $\Var_{H_0}[GHC(t)]=1$. Hence, 
	$$\P_{H_0}\left(\max_{t\in \left\{\sqrt{2L\log{d}}: ~ L \in (0,5) \right\} \bigcap \mathbb{N}} |GHC(t)| > \log d\right)\leq \frac{\sqrt{5\log{d}}}{\log^2{d}} \rightarrow 0,$$
	which shows that the probability of Type I error goes to zero. 
	
	Next, we consider the probability of Type II error. To this end we will show that $GHC(t_r)$ beats the cut-off of $\log{d}$ with high probability. However, $t_r$ might not always be an integer (and hence not automatically a member of $\{\sqrt{2L\log{d}}: L\in (0,5)\}\cap \mathbb{N}\}$). But, our proof goes through for any $GHC(t_r')$ whenever $t_r'=(1+o(1))t_r$ and hence, the result will also hold for $GHC(\lceil t_r\rceil)$. Therefore, to keep notation simple we only show that $GHC(t_r)$ beats the cut-off of $\log{d}$ with high probability. For this, suppose $\bm p \in \mathcal{P}(U([d]), s, \varepsilon)$ and $S \subseteq \{1, 2, \ldots, d\}$ be the subset where the $\bm p$ differs from $1/d$. Then 
	\begin{align}
	GHC(t_r) = T_1(t_r)+ T_2(t_r), 
	\end{align}
	where 
	\begin{align}\label{eq:T1t}
	T_1(t_r) & := \frac{\sum_{j \in S} \{\bm 1\{ |D_j| \geq t_r \} - \P_{H_0}(|D_j| \geq t)\}}{\sqrt{\sum_{j=1}^d  \P_{H_0}(|D_j| \geq t)(1-\P_{H_0}(|D_j| \geq t_r))}}  \\ 
	T_2(t_r) & := \frac{\sum_{j \notin S} \{\bm 1\{ |D_j| \geq t_r \} - \P_{H_0}(|D_j| \geq t_r)\}}{\sqrt{\sum_{j=1}^d  \P_{H_0}(|D_j| \geq t_r)(1-\P_{H_0}(|D_j| \geq t_r))}} . \label{eq:T2t}
	\end{align} 
	Note that, for any $t \in \R$ and $T_2(t)$ as defined in \eqref{eq:T2t} above, $\E_{\bm p}[T_2(t)]=0$ and $\Var_{\bm p}[T_2(t)]=1$. Hence, for any $t \in \R$, $\P_{\bm p}(|T_2(t)| > 2 \log d)\rightarrow 0$. Therefore, the power of the HC test can be bounded below through: 
\begin{align*}
    \P_{\bm p}(|GHC(t_r)| > \log d) \geq \P_{\bm p}(T_1(t_r) > 2\log d \text{ and } |T_2(t_r)| \leq \log d)\rightarrow 1,
\end{align*}
	where the last step uses Lemma \ref{lm:T1_S} below.

	\begin{lem}\label{lm:T1_S}
		$\P_{\bm p}(T_1(t_r) > 2 \log d ) \rightarrow 1$. 
	\end{lem}

	\begin{proof} 
	We claim that to prove Lemma \ref{lm:T1_S} it suffices to show the following two facts: 
		\begin{align}\label{eq:D1_threshold}
		\inf_S \E_{\bm p} \left[ \sum_{j \in S} \bm 1\{ |D_j| \geq t_r \}  \right] \gg s \P_{H_0}(|D_1| \geq t_r),  
		\end{align} 
			and 
			\textcolor{black}{\begin{align}\label{eq:T1_I}
		\inf_S \E_{\bm p}\left[T_1(t_r) \right] \gtrsim d^\eta, \quad \text { for some } \eta > 0. 
		\end{align} }
 	
To see this, note that by a Cramer-type moderate deviation inequality for independent sums \cite[Chapter 8]{petrov1975independent}, whenever $t \ll n^{\frac{1}{6}}$, 
	\begin{align}\label{eq:D1_tail}		
			\P_{H_0}(D_1 \geq t) = (1+o(1))\bar \Phi(t) \quad \text{and} \quad  \P_{H_0}(D_1 \leq - t) = (1+o(1))\bar \Phi(t) , 
			\end{align} 
	 where $\bar \Phi(t)= 1-\Phi(t)$ is the upper tail of the standard normal distribution. This implies, since $t_r = \sqrt{2 r \log d}$,  
\begin{align}\label{eq:D1}
\P_{H_0}(|D_1| \geq t_r)(1-\P_{H_0}(|D_1| \geq t_r)) & = 2  (1+o(1)) \bar \Phi(t_r) (1-2\bar \Phi(t_r)) \nonumber \\ 
& = (1+o(1))  \frac{2  \phi(t_r)}{\sqrt{t_r}}  \gtrsim \frac{1}{d^r \log^{\frac{1}{4}} d} . 
\end{align}
Hence,  
		\begin{align}
		\Var_{\bm p}[T_1(t_r)] \leq  &  \frac{\E_{\bm p} \left[ \sum_{j \in S} \bm 1\{ |D_j| \geq t_r \} \right] }{ d \P_{H_0}(|D_1| \geq t_r)(1-\P_{H_0}(|D_1| \geq t_r)) } \nonumber \\ 
		& = \frac{(1+o(1)) \left(\E_{\bm p} \left[ \sum_{j \in S} \bm 1\{ |D_j| \geq t_r \} \right] - s \P_{H_0}(|D_1| \geq t_r ) \right) }{ d \P_{H_0}(|D_1| \geq t_r)(1-\P_{H_0}(|D_1| \geq t_r)) } \tag*{(by \eqref{eq:D1_threshold})} \nonumber \\ 
		& = \frac{(1+o(1)) \E_{\bm p}[T_1(t_r)]}{\sqrt{d \P_{H_0}(|D_1| \geq t_r)(1-\P_{H_0}(|D_1| \geq t_r)) }} \tag*{(recall definition from \eqref{eq:T1t})} \nonumber \\
		& \lesssim \E_{\bm p}[T_1(t_r)] d^{\frac{r}{2}-\frac{1}{2}}\log^{\frac{1}{8}} d, \nonumber 
		\end{align} 
	where the last step uses \eqref{eq:D1}. Hence, using \eqref{eq:T1_I}, for $d$ large enough, 
		\begin{align}
		\P_{\bm p}(T_1(t_r) \leq 2 \log d ) \lesssim  \frac{\Var_{\bm p}\left[T_1(t_r) \right] }{(\E_{\bm p}[T_1(t_r)])^2} \lesssim \frac{  \log^{\frac{1}{8}} d}{d^{\frac{1}{2}-\frac{r}{2} + \eta} \E_{\bm p}[T_1(t_r)] }  \ll 1, \nonumber 
		\end{align}
		since $r \leq 1$ and $\eta > 0$. This shows Lemma \ref{lm:T1_S} holds, whenever \eqref{eq:T1_I} and 
		\eqref{eq:D1_threshold} hold.  
\end{proof} \medskip
		
		\noindent{\it Proof of \eqref{eq:D1_threshold}}: Throughout we will assume that $\max_{j \in [d]} |\Delta_j| \leq \sqrt{2 \log d/n}$. 
		For $j \in [d]$, denote 
		$$D_j'=\frac{Z_j - np_j}{\sqrt{n p_j}}.$$ Then, recalling \eqref{eq:D}, 
		\begin{align}\label{eq:ZD}
		D_j= \frac{Z_j -n/d}{\sqrt{n/d}} =  D_j'  \sqrt{ d p_j }  -  \Delta_j \sqrt{nd}, 
		\end{align} 
		since $\Delta_j=\frac{1}{d}-p_j$. Note that $$\max_{j \in [d]}|d p_j-1| = d \max_{j \in [d]}|\Delta_j| = O\left(\sqrt{\frac{d \log d}{n}}\right) = o(1).$$ Define $\bar \varepsilon_j = |\Delta_j| \sqrt{\frac{nd}{2 \log d}} $, for $j \in [d]$. {Then by a Poisson moderate deviation-type inequality \cite[Chapter 8]{petrov1975independent},}
		\begin{align}
		\E_{\bm p} & \left[ \sum_{j \in S} \bm 1\{|D_j| \geq  t_r\} \right] \nonumber \\ 
		& = (1+o(1))   \sum_{j \in S}  \left\{ \bar \Phi\left(\frac{t_r-\Delta_j \sqrt{nd}}{ \sqrt{ d p_j } } \right) + \bar \Phi\left(\frac{t_r+\Delta_j \sqrt{nd}}{ \sqrt{ d p_j } } \right) \right\} \quad (\text{with the } 1+o(1) \ \text{being free of $S$})\nonumber \\ 
		& = (1+o(1))   \sum_{j \in S}  \left\{ \bar \Phi\left(\sqrt{\frac{ 2 \log d}{ d p_j } } \left(\sqrt r -\Delta_j \sqrt{\frac{nd}{2 \log d}} \right) \right) +  \bar \Phi\left(\sqrt{\frac{ 2 \log d}{ d p_j } } \left(\sqrt r + \Delta_j \sqrt{\frac{nd}{2 \log d}} \right) \right) \right\} \nonumber \\ 
		& = (1+o(1))   \sum_{j \in S}  \left\{ \bar \Phi\left(\sqrt{\frac{ 2 \log d}{ d p_j } } \left(\sqrt r - \bar \varepsilon_j \right) \right) +  \bar \Phi\left(\sqrt{\frac{ 2 \log d}{ d p_j } } \left(\sqrt r + \bar \varepsilon_j \right) \right) \right\} \nonumber \\ 
		& = (1+o(1))   \sum_{j \in S}  \left\{ \bar \Phi\left(\sqrt{2 \log d}  \left(\sqrt r - \bar \varepsilon_j \right) \right) +  \bar \Phi\left(\sqrt{ 2 \log d } \left(\sqrt r + \bar \varepsilon_j \right) \right) \right\},\nonumber 
		\end{align}  
		where the last step uses Observation \ref{obs:normal_tail}. This implies, 
		\begin{align}\label{eq:D1_function}
	\inf_S \E_{\bm p}  \left[ \sum_{j \in S} \bm 1\{|D_j| \geq  t_r \} \right] & = (1+o(1))\inf_{S} \sum_{j\in S}\left\{\bar{\Phi}(t_r-y_j)+\bar{\Phi}(t_r+y_j)\right\}. 
		\end{align} 
where $y_j:=\sqrt{2 \bar \varepsilon_j \log{d} }$, for $j \in [d]$. Therefore, to show \eqref{eq:D1_threshold} we have to lower bound the variational problem in the RHS of \eqref{eq:D1_function} such that the constraint $\|\bm p- U([d])\|_1\geq s\sqrt{2C^*\log{d}/(nd)}$ is satisfied. This constraint can be written as 
\begin{align*}
 \sum_{j\in S}y_j\geq s\sqrt{2C^*\log{d}}, \quad \text{ where } 0<y_j\leq \sqrt{2\log{d}}, \text{ for } j \in S . 
 \end{align*} 
To this end, we appeal to the strategy employed in \cite[Lemma 6.2 and Lemma 7.4]{ingster2010detection}. To operationalize the argument, consider the function 
\begin{align*}
F_r(y)=\bar{\Phi}(t_r-y)+\bar{\Phi}(t_r+y),
\end{align*}
where $0<y\leq \sqrt{2\log{d}}$. The result in \eqref{eq:D1_threshold} will then follow from the following lemma:

\begin{lem}\label{lm:F}
If there exists $\lambda>0$ such that
		\begin{align}
		\inf_{y\in (0,\sqrt{2\log{d}}]}\left(F_r(y)-\lambda y\right)=F_r(\sqrt{2C^*\log{d}})-\lambda\sqrt{2C^*\log{d}}.\label{eqn:hc_optimization}
		\end{align} 
then the following holds: 
\begin{align*}
    \inf\left\{\sum_{j\in S}\left(\bar{\Phi}(t_r-y_j)+\bar{\Phi}(t_r+y_j)\right): 0\leq y_j\leq \sqrt{2\log{d}},
    \sum_{j\in S} y_j\geq s\sqrt{2C^*\log{d}}\right\}=sF_r(\sqrt{2C^*\log{d}})
\end{align*}
\end{lem} 

%

\noindent {\it Proof of} \eqref{eqn:hc_optimization}: To prove \eqref{eqn:hc_optimization}, 
let $$G_{\lambda}(y)=F_r(y)-\lambda y$$ and hence, $G'_{\lambda}(y)=F'_r(y)-\lambda $. Now, if we want $G'_{\lambda}(\hat y)=0$ at $\hat y:=\sqrt{2C^*\log{d}}$, then $\lambda=\hat \lambda:=\phi(t_r-\hat y)-\phi(t_r+\hat y)$. This is a feasible choice, since by a direct calculation it can be checked that $\liminf \hat \lambda>0$. To show that this is choice of $\hat y$ is indeed a global minimum of $G_{\hat \lambda}(y)$ in \eqref{eqn:hc_optimization}, we will next divide our analysis in two cases.
\\

\noindent{\it Case} 1: $r=1$. In this case, we can safely assume $C^*\leq 1$. 
Now, note that 
\begin{align*}
\inf_{y\in (0,\sqrt{2\log{d}}]}G''_{\hat \lambda}(y)&=\inf_{y\in (0,\sqrt{2\log{d}}]}\left\{(t_r+y)\phi(t_r+y)+(t_r-y)\phi((t_r-y))\right\}>0, 
\end{align*} 
uniformly in $n,d,s$, since $t_r-y=\sqrt{2\log{d}} -y \geq 0$, using $C^*\leq 1=r$. This shows, in the case $r=1$, the function $G_{\hat \lambda}(y)$ is concave in the domain $y\in (0,\sqrt{2\log{d}}]$ and  hence, $\hat y$ is a global minima of $G_{\hat \lambda}(y)$ in this domain. \\ 

\noindent{\it Case} 2: $r=4C^*$. In this case $G''_{\hat \lambda}(y)$ can potentially be negative at some values $0<y < \sqrt{2\log{d}}$ and hence, a direct convexity argument does not work. We therefore need to study the function $G_{\hat \lambda}$ a little more closely. To this end, first note that $\inf_{y\in (0,t_r]}G''_{\hat \lambda}(y)>0$. Hence, $\hat y$ is the global minimum of $G_{\hat \lambda}(\cdot)$ over the sub-domain $(0,t_r]$. Next, note that for any $y>t_r$, 
\begin{align}\label{eq:G_lambda}
G_{\hat \lambda}(\hat y)-G_{\hat \lambda}(y)&=\left[\bar{\Phi}(t_r-\hat y)-\bar{\Phi}(t_r-y)\right]+\left[\bar{\Phi}(t_r+\hat y)-\bar{\Phi}(t_r+y)\right]+\hat \lambda(y-\hat y) \nonumber \\ 
& = \left[\bar{\Phi}(\hat y)-\bar{\Phi}(t_r-y)\right]+\left[\bar{\Phi}(3 \hat y)-\bar{\Phi}(t_r+y)\right]+\hat \lambda(y-\hat y), 
\end{align}
since $t_r = \sqrt{2 r \log d} = \sqrt{8 C^* \log d} = 2 \hat y$. 
Now, we have the following observation:

\begin{obs}\label{obs:G_tr} Fix $0< \delta<\frac{1}{2}$. Then for any $M>0$, exists a $d_{\delta}>0$ large enough such that 
\begin{align}\label{eq:G_tr}
G_{\hat \lambda}(t_r(1+\theta))>G_{\hat \lambda}(\hat y)+\delta,
\end{align}
for $\theta\in [0,\frac{M}{\log{d}}]$ whenever $d\geq d_{\delta}$. Moreover, exists $M'>0$ such that uniformly over $\theta\geq \frac{M'}{\log{d}}$, 
\begin{align}\label{eq:G_hat}
G_{\hat \lambda}''(t_r(1+\theta))<0, 
\end{align} 
for $d$ large enough.  
\end{obs} 

\begin{proof} To begin with, set $y_\theta=t_r(1+\theta)$. Then, since $|y_\theta-t_r|\rightarrow 0$ for  $\theta\in [0,\frac{M}{\log{d}}]$,
$$\bar{\Phi}(\hat y)-\bar{\Phi}(t_r-y_\theta) \rightarrow -\tfrac{1}{2}.$$
Moreover, for any $y>0$, $\bar{\Phi}(3 \hat y)-\bar{\Phi}(t_r+y) \rightarrow 0$. Finally, uniformly for $\theta \in [0,\frac{M}{\log{d}}]$ we have, 
\begin{align}\label{eq:y_hat}
\hat \lambda(y_\theta-\hat y)  \lesssim \left(\phi(t_r-\hat y)-\phi(t_r+\hat y)\right)\sqrt{\log{d}} 
& = \left(\phi(\hat y)-\phi(3 \hat y)\right)\sqrt{\log{d}} \nonumber \\  
& \lesssim \log{d}/d^{C^*} \nonumber \\ 
& \rightarrow 0,
\end{align}
as $d\rightarrow \infty$. Hence, by \eqref{eq:G_lambda}, given $0<\delta<\frac{1}{2}$, there exists $d_{\delta}>0$ such that \eqref{eq:G_tr} holds.  

Now, we prove \eqref{eq:G_hat}. Towards this, note that
\begin{align*}
G_{\hat \lambda}''(t_r(1+\theta))=t_r\left[(2+\theta)\phi(t_r(2+\theta))-\theta\phi(-t_r\theta)\right] 
&=t_r\left[(2+\theta)\phi(t_r(2+\theta))-\theta\phi(t_r\theta)\right] \nonumber \\ 
& =t_r\left[\psi(2+\theta)-\psi(\theta)\right].
\end{align*}
where $\psi(x)=x\phi(t_rx)$. Now, note that there exists $M'>0$ such that $\psi'(x)=\phi(t_r(x))\left[1-(t_rx)^2\right]<0$ uniformly in $x>(\frac{M}{\sqrt{\log{d}}},B]$, for any $B>0$. Therefore, by taking $B=\sqrt{1/r}$ we have $G_{\hat \lambda}''(x)$ is negative for all $x\in [t_r(1+M/\sqrt{\log{d}}),\sqrt{2\log{d}})$, and hence concave in that neighborhood. 
\end{proof}

The result in Observation \ref{obs:G_tr} implies that checking 
 \begin{align}\label{eq:G_value}
  G_{\hat \lambda}(\hat y)<G_{\hat \lambda}(\sqrt{2\log{d}}), 
\end{align} 
for $d$ large enough, will complete the proof of Case 2.. To show this, note that, 
\begin{align*}
G_{\hat \lambda}(\hat y) & -G_{\hat \lambda}(\sqrt{2\log{d}}) \nonumber \\ 
&=\left[\bar{\Phi}(t_r-\hat y)-\bar{\Phi}(t_r-\sqrt{2\log{d}})\right]+\left[\bar{\Phi}(t_r+\hat y)-\bar{\Phi}(t_r+\sqrt{2\log{d}})\right]+\hat \lambda(\sqrt{2\log{d}}-\hat y) \nonumber \\ 
& < \left[\bar{\Phi}(\hat y)-\bar{\Phi}(t_r-\sqrt{2\log{d}})\right]+\left[\bar{\Phi}(3 \hat y)-\bar{\Phi}(t_r+\sqrt{2\log{d}})\right] +\hat \lambda(\sqrt{2\log{d}}-\hat y).
\end{align*}
Note that, as $d \rightarrow \infty$,  the first term on the RHS above converges to $-1$, the second term converges to $0$, and the third term converge to $0$ by arguments as in \eqref{eq:y_hat}). This implies \eqref{eq:G_value} for $d$ large enough.  \\

\noindent \noindent{\it Proof of  \eqref{eq:D1_threshold} and \eqref{eq:T1_I}}: First for the proof of \eqref{eq:D1_threshold} we note from Lemma \ref{lm:F} and \eqref{eq:D1_function} that
\begin{align*}
    \inf _{\bm p}\E_{\bm p} 
    \left[ \sum_{j \in S} \bm 1\{|D_j| \geq  t_r\} \right]& = (1+o(1)) \inf_{S} \sum_{j\in S}\left\{\bar{\Phi}(t_r-y_j)+\bar{\Phi}(t_r+y_j)\right\}\\
    &\geq (1+o(1))sF_r(\sqrt{2C^*\log{d}}) \\
    & \gg s\P_{H_0}(|D_1|\geq t_r),
\end{align*}
by direct calculations using Mill's ratio estimates and $r=\min\{4C^*,1\}$. This completes the proof of  \eqref{eq:D1_threshold}.

Next,
	\begin{align*}
		\E_{\bm p}[T_1(t_r)] &= \frac{\E_{\bm p} \left[ \sum_{j \in S} \bm 1\{ |D_j| \geq t_r \} \right]- s \P_{H_0} (|D_1| \geq t_r ) }{ \sqrt{d \P_{H_0}(|D_1| \geq t_r)(1-\P_{H_0}(|D_1| \geq t_r))}}\\
		&\geq \frac{(1+o(1))s\left\{ F_r(\sqrt{2C^*\log{d}})-\P_{H_0}(|D_1|\geq t_r)\right\} }{ \sqrt{d \P_{H_0}(|D_1| \geq t_r)(1-\P_{H_0}(|D_1| \geq t_r))}}\\
		& \gtrsim d^\eta, \quad \text { for some } \eta > 0,
	\end{align*}
	by calculations exactly parallel to the proof of \cite[Lemma 6.4 (a)]{mukherjee2018detection}. This completes the proof of \eqref{eq:T1_I}.
	

\subsubsection{Lower Bound for Theorem \em{\ref{thm:one_sample_alphageqhalf} (1)}}
	\label{sec:geqhalf_lower_bound_pf}
	
In this section we prove part (b) of Theorem \ref{thm:one_sample_alphageqhalf} (1), that is, the lower bound on the minimax risk in the regime $\alpha \leq \frac{1}{2}$ and $n \gtrsim d^{\frac{1}{2}+\alpha}$. Our proof uses the truncated second moment method of Ingster (as presented in \cite{butucea2013detection}) based on a suitable prior $\pi_n$ on  $\cP(U([d]), s, \varepsilon)$. To describe the prior assume, without loss of generality, $d$ and  $s$ are even. Choose a subset $S$ of size $s/2$ uniformly at random from the first half of the domain $\cD:=\{1, 2, \ldots, d/2\}$ and another subset $S'$ of size $s/2$ uniformly at random from the second half of the domain $\cD':=\{d/2+1, \ldots, 1\}$. Recall the definition of $C(\alpha)$ from \eqref{eq:threshold_III}. Throughout, we set $\eta^2= 2 C(\alpha)(1-\delta)\frac{\log{d}}{nd}$, for $0 < \delta < 1$ fixed, and define $\bm p^{S, S'} := (p_1, p_2, \ldots, p_d) \in U([d])$, where 
$$p_j= 
	\left\{
	\begin{array}{cc}
	\frac{1}{d} + \eta  &  \text{for }  j \in S, \\ 
	\frac{1}{d} - \eta &  \text{for } j \in S', \\
	\frac{1}{d} & \text{otherwise}. 
	\end{array}
	\right.
$$ 	
Clearly, $\bm p^{S, S'} \in \cP(U([d]), s, \varepsilon)$. To operationalize a truncated second moment argument we next introduce the good event
	\begin{align}
	\mathcal{G}=\left\{\sup_{1 \leq j \leq d}|D_j|\leq \sqrt{2\log{d}}\right\},\label{eqn:good_event_trunc_second_moment}	
	\end{align}
where $D_j$, for $j \in [d]$, is as defined in \eqref{eq:D}. Note that under the null $D_1, D_2, \ldots, D_d$ are i.i.d. $\dPois(n/d)$. Hence, by an union bound and \eqref{eq:D1_tail}, 
\begin{align}\label{eq:G} 
\P_{H_0}(\cG^c)  \lesssim d \P_{Z_1 \sim \dPois(n/d)}(|Z_1|> \sqrt{2 \log d}) & \lesssim  (1+o(1)) d \bar \Phi(\sqrt{2 \log d}) \nonumber \\ 
& \lesssim   (1+o(1)) \frac{d\phi(\sqrt{2 \log d})}{\sqrt{\log d}} \nonumber \\ 
& \rightarrow  0. 
\end{align}
Therefore, as in \cite[Section 5]{butucea2013detection}, to show the result in part (b) of Theorem \ref{thm:one_sample_alphageqhalf} (1), it suffices to prove the following estimates: 
$$\mathbb{E}_{H_0}[L_{\pi_n}\mathbf{1}\{\mathcal{G}\}]=1+o(1) \quad \text{and} \quad \mathbb{E}_{H_0}[L_{\pi_n}^2\mathbf{1}\{\mathcal{G}\}] \leq 1 + o(1).$$
These estimates are proved below in Lemma \ref{lm:LpiG} and Lemma \ref{lm:Lpi2G}, respectively.  


\begin{lem}\label{lm:LpiG} For $\pi_n$ and $\cG$ as defined above, $\mathbb{E}_{H_0}[L_{\pi_n}\mathbf{1}\{\mathcal{G}^c\}]=o(1)$. 
\end{lem}

\begin{proof} To begin with, note that a simple change of measure argument gives, 
	\begin{align*}
	\mathbb{E}_{H_0}[L_{\pi_n}\mathbf{1}\{\mathcal{G}^c\}]&=\frac{1}{{d/2 \choose s/2}^2}\sum\limits_{S\in \cD, S'\in \cD'} \mathbb{P}_{\bm p^{S,S'}}(\mathcal{G}^c),
	\end{align*}
	where 
	\begin{align*}
	\mathbb{P}_{\bm p^{S,S'}}(\mathcal{G}^c):=\mathbb{P}_{\bm p^{S,S'}}\left(\max_{1 \leq j \leq d}|D_j|> \sqrt{2 \log d} \right) & \leq T_1 + T_2,
	\end{align*}
with 
\begin{align}\label{eq:T12_HC}
T_1 := \sum_{j \in (\cD \setminus S) \cup (\cD' \setminus S')}\P_{p_j}(|D_j|> \sqrt{2 \log d}) \quad \text{and} \quad T_2:= \sum_{j \in S\cup S'}\P_{p_j}(|D_j|> \sqrt{2 \log d}).
\end{align}

To begin with consider $T_1$. Note that, since for $j \in (\cD \setminus S) \cup (\cD' \setminus S')$ $Z_j \sim \dPois(n/d)$, by hence by arguments similar to \eqref{eq:G} it is immediate that  
\begin{align}\label{eq:T1_bound}
T_1  \lesssim d \P_{Z_1 \sim \dPois(n/d)}(|Z_1|> \sqrt{2 \log d}) \rightarrow  0. 
\end{align}


Next, consider $T_2$ (recall definition from \eqref{eq:T12_HC}. Note that 
\begin{align}\label{eq:T212}
T_2 =  \sum_{j \in S}\P_{p_j}(|D_j|> \sqrt{2 \log d}) + \sum_{j \in S'}\P_{p_j}(|D_j|> \sqrt{2 \log d}) := T_{12} + T_{22}. 
\end{align}
We begin with $T_{12}$. To this end, define $$D_j'=\frac{Z_j - np_j}{\sqrt{n p_j}}.$$ Then for $j \in S$, recalling \eqref{eq:ZD} gives, 
	\begin{align*}
	\P_{p_j}(|D_j|>\sqrt{2 \log d})&=\P_{p_j}\left(D_j'>  \sqrt{\frac{2 \log d}{d p_j}}-\frac{n \Delta_j}{\sqrt{n p_j}}\right)+
	\P_{p_j}\left( D_j'<-\sqrt{\frac{2 \log d}{d p_j}}-\frac{n \Delta_j}{\sqrt{n p_j}}\right) , 
	\end{align*}
where $\Delta_j = \frac{1}{d} - p_j$. Now, note that, for $j \in S$, $|d p_j -1 |= \sqrt{d\log d/n} =o(1)$, since $n \gg d \log d$, by assumption of Theorem \ref{thm:one_sample_alphageqhalf} (1). Therefore, uniformly for $j \in S$, 
	\begin{align*} 
\sqrt{\frac{2 \log d}{d p_j}} = (1+o(1)) \sqrt{2 \log d} \quad \text{and} \quad 
	\frac{n \Delta_j}{\sqrt{n p_j}} = -(1+o(1))\sqrt{2C_*\log{d}},
	\end{align*}
	where $C_*:=C(\alpha)(1-\delta)$. Hence, recalling the definition of  $T_{12}$ from \eqref{eq:T212} gives, 
	\begin{align}\label{eq:T21_tail}
	 T_{12} :=\sum_{j\in S} \left\{\P_{p_j}\left(D_j'> A\sqrt{2\log{d}}(1+o(1))\right) +  \P_{p_j}\left(D_j'< B\sqrt{2\log{d}})(1+o(1) \right) \right\}, 
	\end{align}
	where $A:=\sqrt{C_*}+1$ and $B:=\sqrt{C_*}-1$. Now, by a moderate deviation bound for the Poisson distribution (see \cite[Lemma 2]{arias2015sparse}), 
	\begin{align}\label{eq:T21_tail1}
\sum_{j\in S}	\P_{p_j}\left(D_j'> A \sqrt{2 \log{d}} (1+o(1))\right) \lesssim s e^{- A^2 \log{d} (1+o(1))} \rightarrow 0, 
	\end{align}
as $d \rightarrow \infty$. Similarly, by \cite[Lemma 2]{arias2015sparse} for the lower tail, 
	\begin{align}\label{eq:T21_tail2}
\sum_{j\in S}	\P_{p_j}\left(D_j' < B \sqrt{2 \log{d}} (1+o(1))\right) \lesssim s e^{- B^2 \log{d}(1+o(1))} \rightarrow 0, 
	\end{align}
where the last limit follows since $B:=(1-\sqrt{C}_*)>1-\alpha$. Combining, \eqref{eq:T21_tail}, \eqref{eq:T21_tail1}, and \eqref{eq:T21_tail2}, now gives $T_{21} \rightarrow 0$. A similar argument shows that $T_{22} \rightarrow 0$. Therefore, by \eqref{eq:T212}, $T_2 \rightarrow 0$. This together with \eqref{eq:T1_bound},  shows that uniformly in $S,S'$ the following hold: 
$$\P_{S,S'}(\mathcal{G}^c)\rightarrow 0.$$
This immediately implies that $\E_{H_0}[L_{\pi_n}\mathbf{1}\{\mathcal{G}^c\}] \rightarrow 0$, completing the proof of Lemma \ref{lm:LpiG}. 
\end{proof}

Next, we first consider the truncated second moment of $L_{\pi_n}$. 

\begin{lem}\label{lm:Lpi2G} For $\cG$ as defined in  \eqref{eqn:good_event_trunc_second_moment}, $\mathbb{E}_{H_0}[L_{\pi_n}^2 \mathbf{1}\{\mathcal{G}\}] \leq 1+ o(1)$. 
\end{lem}

\begin{proof} To begin with, note that 
\begin{align}\label{eq:Lpi2_HC}
\E_{H_0}\left[L_{\pi_n}^2 \mathbf{1}\{\mathcal G\} \right] &\leq \frac{1}{{d/2 \choose s/2}^4}\sum_{S_1,S_2 \in \cD,\atop S_1',S_2' \in \cD'} \prod_{j =1}^d \E_{H_0} \left[  e^{n \Delta_j + Z_j\log(d p_j)} \mathbf{1}\left\{|D_j|\leq \sqrt{2 \log d} \right\} \right] . 
\end{align}
Now, for $j \in (S_1\Delta S_2) \cup (S_1'\Delta S_2')$ by  a simple of change of measure we have, 
\begin{align*}
\E_{H_0}\left[ e^{n \Delta_j + Z_j\log(d p_j)} \mathbf{1}\left\{|D_j|\leq \sqrt{2 \log d} \right\} \right] &=\P_{p_j}(|D_j|\leq \sqrt{2 \log d})\leq 1.
\end{align*}
Consequently, \eqref{eq:Lpi2_HC} and a direct calculation yields that
\begin{align}\label{eq:Lpi2_HC_S}
\E_{H_0} & \left[L_{\pi_n}^2 \mathbf{1}\{\mathcal G\} \right] \nonumber \\ 
&\leq \frac{1}{{d/2 \choose s/2}^4}\sum_{S_1,S_2 \in \cD,\atop S_1',S_2' \in \cD'}  \prod_{j \in (S_1\cap S_2)\cup (S_1'\cap S_2')} \E_{H_0} \left[e^{2 n \Delta_j + 2 Z_j\log(d p_j)}  \mathbf{1}\left\{|D_j|\leq \sqrt{2 \log d}\right\} \right] \nonumber \\ 
&= \frac{1}{{d/2 \choose s/2}^4}\sum_{S_1,S_2 \in \cD,\atop S_1',S_2' \in \cD'}  \prod_{j \in (S_1\cap S_2)\cup (S_1'\cap S_2')} e^{2C_*\log{d}} \P_{D_j \sim \dPois(n d p_j^2)}(|D_j|\leq \sqrt{2 \log d}), 
\end{align} 
where $C_*:=C(\alpha)(1-\delta)$. Now, we consider the two cases depending on the value of $C_*$. \\

\noindent{\it Case} 1: $4C_*<1$. This implies, $C_*<\min\{C(\alpha), \frac{1}{4} \} \leq \alpha-\frac{1}{2}$, recalling the definition of $C_*$ from \eqref{eq:threshold_III}. 
Now, bounding $\P_{Z_j \sim \dPois(n d p_j^2)} (|D_j|\leq t_d)$ by $1$ we get from \eqref{eq:Lpi2_HC_S}, 
\begin{align*} 
\E_{H_0}\left[L_{\pi_n}^2 \mathbf{1}\{\mathcal G\} \right]&\leq \E\left[e^{2C_*\log{d}|S_1\cap S_2|} \right] \E\left[ e^{2C_*\log{d}|S_1'\cap S_2'|} \right],
\end{align*}
where the above expectations are with respect to the randomness of $|S_1\cap S_2|$ and $|S_1'\cap S_2'|$. Note that $|S_1 \cap S_2|$ and  $|S_1' \cap S_2'|$ are distributed as independent $\mathrm{Hypergeometric}(d/2, s/2, s/2)$, which is dominated by the $\dBin(s/2, s/d)$ distribution in convex ordering \cite[Proposition 20.6]{aldous1985exchangeability}. Hence, 
\begin{align*} 
\E_{H_0}\left[L_{\pi_n}^2 \mathbf{1}\{\mathcal G\} \right]& \leq \exp\left\{s \log \left(1 + 
\frac{s}{d} \left(e^{2C_*\log{d}} - 1 \right) \right) \right\} \\
& \leq \exp\left\{\frac{s^2}{d}\left(e^{(1+o(1)) 2C_*\log{d}}-1\right)\right\}\\
&\leq \exp\left\{e^{(1+o(1)) (2C_*-(2\alpha-1))  \log{d} } \right\}\\
&=1+o(1),
\end{align*}
where the last step follows using $C_*<\alpha-\frac{1}{2}$. 
\\ 

\noindent{\it Case} 2: $4C_* \geq 1$. Note that this is only possible for $\alpha>\frac{3}{4}$. We now have to estimate $\P_{Z_j \sim \dPois(n d p_j^2)} (|D_j|\leq \sqrt{2 \log d})$, for $j \in (S_1\cap S_2)\cup (S_1'\cap S_2')$. To do this, first a direct calculation shows that for $j \in (S_1\cap S_2)\cup (S_1'\cap S_2')$, 
\begin{align*}
\frac{n/d-n d p_j^2}{\sqrt{n d p_j^2}}=- (1+o(1)) 2\sqrt{2C_*\log{d}} \quad \text{and} \quad  \sqrt{\frac{(2 \log d) n/d}{n dp_j^2 }}=1+o(1).
\end{align*}
Moreover, since $4C_*>1$, $\sqrt{2\log d}<2\sqrt{2C_*\log{d}}$. Hence, by \cite[Lemma 2]{arias2015sparse}, 
\begin{align*}
 \P_{Z_j \sim \dPois(n d p_j^2)} & (|D_j| \leq \sqrt{2 \log d}) \\ 
&\leq \P_{Z_j \sim \dPois(n d p_j^2)} \left(\frac{Z_j-n d p_j^2}{\sqrt{n d p_j^2}}\leq (1+o(1)) \left(\sqrt{2}-2\sqrt{2C_*} \right) \sqrt{\log{d}} \right) \\ 
&\leq \exp\left\{- (1+o(1)) \left(1-2\sqrt{C_*}\right)^2 \log{d}\right\}
\end{align*}
Therefore, from \eqref{eq:Lpi2_HC_S}, we have
\begin{align*} 
\E_{H_0}\left[L_{\pi_n}^2 \mathbf{1}\{\mathcal G\} \right]
&\leq \left(\E\left[\exp\left\{(1+o(1)) \left(2C_*-\left(1-2\sqrt{C_*}\right)^2 \right)\log{d}|S_1\cap S_2|\right\}\right] \right)^2
\end{align*}
Now, note that $2C_*-\left(1-2\sqrt{C_*}\right)^2>0$ whenever $1>C_*>\frac{1}{4}$. Hence, by the Hypergeometric-Binomial convex ordering argument \cite[Proposition 20.6]{aldous1985exchangeability} we have, 
\begin{align*}
\E_{H_0}\left[L_{\pi_n}^2 \mathbf{1}\{\mathcal G\} \right] &\leq \left(\exp\left\{ \frac{s^2}{d}\exp\left\{ (1+o(1))  \left(2C_*-\left(1-2\sqrt{C_*}\right)^2\right) \log{d}\right\}\right\}\right)^2 \\
&=\left(\exp\left[\exp\left\{(1+o(1)) \left(2C_*-\left(1-2\sqrt{C_*}\right)^2 -(2\alpha-1) \right) \log{d} \right\}\right] \right)^2\\
&=1+o(1),
\end{align*}
since $C_*< C(\alpha) = (1-\sqrt{1-\alpha})^2$, for $\alpha > \frac{3}{4}$. This completes the proof of the Lemma \ref{lm:Lpi2G}. 
\end{proof}

%

	\subsection{Lower Bound in the Impossibility Regime}
	\label{sec:lower_bound_pf}
	
In this section we consider the regime where no tests are powerful irrespective of the value of $\varepsilon$. This includes two cases: (a) when $\alpha \leq \frac{1}{2}$ and $n \ll d^{\frac{1}{2}+\alpha}$ (Theorem \ref{thm:one_sample_alphaleqhalf} (2)) and (b) $\alpha > \frac{1}{2}$ and $n \ll d \log d$ (Theorem \ref{thm:one_sample_alphageqhalf} (2)).  

	\begin{ppn}\label{ppn:threshold_lb_II} Suppose either one of the following two conditions hold: 
		\begin{enumerate}
			\item[(a)] $\alpha \leq \frac{1}{2}$ and $n \ll d^{\frac{1}{2}+\alpha}$, or 
			
			\item[(b)] $\alpha > \frac{1}{2}$ and $n \ll d \log d$.
		\end{enumerate}  
		 Then  all tests are asymptotically powerless if $\limsup \frac{d\varepsilon}{s}\leq 2$.
	\end{ppn}

	\noindent {\it Proof of Proposition} \ref{ppn:threshold_lb_II}: 
Fix $\delta \in (0, 1)$. Let $\{\eta_j: 1 \leq j \leq  d\}$ be i.i.d. $\dBer(t/s)$, where $t:= \delta s$.   
	Then consider a random probability measure $\bm p=(p_1, p_2, \ldots, p_d) \in \cP([d])$ as follows: First choose a subset $S$, with $|S|=s-\lfloor \delta s \rfloor$, of $\{1, 2, \ldots, d/2\}$ uniformly at random,  and let 
	$$p_j= 
	\left\{
	\begin{array}{cc}
	\frac{1}{d} +  \eta_j\frac{(1-\delta) (s-t) }{td} - (1-\eta_j) \frac{1-\delta}{d}   &  \text{for }  j \in S, \\ 
	\frac{1}{d}  &  \text{for } j \in \{1, 2, \ldots, d/2\} \bigcap S^c. 
	\end{array}
	\right.$$ 	
(In other words, for $j \in S$, $p_j=\frac{1}{d}(1+\frac{(1-\delta)^2}{\delta})$ with probability $\delta$ or $p_j = \frac{\delta}{d}$ with probability $1-\delta$, and $p_j=\frac{1}{d}$  for $j \in \{1, 2, \ldots, d/2\} \bigcap S^c$). This defines $\bm p$ for the first half of the first half of the domain. To define $\bm p$ for the second half of the domain, let 
$$\Delta(\bm \eta) := \frac{1-\delta}{d} \sum_{j \in S} \left(\eta_j \frac{s}{t} - 1\right).$$ 
	Note that $-\frac{(1-\delta) (s- \lfloor \delta s \rfloor)}{d} \leq \Delta(\bm \eta)  \leq \frac{(1-\delta) s}{d} \left( \frac{s}{t} - 1 \right)$.  Next, fix a sequence $1 \ll \gamma_d \ll \frac{\sqrt s}{\log d}$ and define the event 
	$$\cG=\left\{\bm \eta: |\Delta(\bm \eta)| \leq \frac{\gamma_d s}{d \sqrt t} \right\}.$$ 
	Then consider the following cases: 
	\begin{itemize}
		
		\item If $\bm \eta \in \cG$, then choose a subset $T$, with $|T| = \lfloor \delta s \rfloor$ uniformly at random from $\{d/2+1, \ldots, d\}$, and define 
		$$p_j= 
		\left\{
		\begin{array}{cc}
		\frac{1}{d} + \frac{|\Delta(\bm \eta)|}{\lfloor \delta s \rfloor}   &  \text{for }  j \in T \text{ and } \Delta(\bm \eta) < 0 , \\ 
		\frac{1}{d} - \frac{|\Delta(\bm \eta)|}{\lfloor \delta s \rfloor}  &  \text{for } j \in T \text{ and } \Delta(\bm \eta) > 0, \\  
		\frac{1}{d}  &  \text{for } j \in \{d/2+1, \ldots, d\} \bigcap T^c. 
		\end{array}
		\right.$$
		
		\item If $\bm \eta \in \cG^c$, then choose $r_n \ll \min\{ \frac{1}{d}, \frac{1}{\sqrt{nd}}\}$  and define 
		$$p_j= 
		\left\{
		\begin{array}{cc}
		\frac{1}{d} + r_n &  \text{for }  j=d/2+1, \\ 
		\frac{1}{d} -  r_n &  \text{for } j = d/2+2, \\  
		\frac{1}{d}     &  \text{otherwise}. 
		\end{array}
		\right.$$
	\end{itemize}
{Note that by construction $\sum_{j=1}^d p_j=1$. The following lemma shows that $\bm p$ belongs to $\cP(U([d]), s, \varepsilon)$ with  high probability, for any  $\varepsilon$ such that 
 $\limsup \frac{d\varepsilon}{s} \leq 2$.}
	
	\begin{lem} For any $0 < \delta < 1$, 
		$$\lim_{n \rightarrow \infty}\P\left(\bm p  \in \cP\left(U[d], s, \frac{2(1-\delta)^3 s}{d} \right) \right) = 1.$$ 
	\end{lem}
	
	\begin{proof}  To begin with note that $\E[\Delta(\bm \eta)]=0$ and $\Var[\Delta(\bm \eta)]= O_\delta(\frac{s^2}{t d^2})$. Then 
		\begin{align}\label{eq:barG}
		\P(\cG^c) \lesssim_{\delta} \frac{d^2 t \Var[\Delta(\bm \eta)]}{ \gamma_d^2 s^2 } \lesssim_{\delta} \frac{1}{ \gamma_d^2 } \ll 1, 
		\end{align}
		since $\gamma_d  \gg 1$. 
		
		Next, note that $\Var\left[||\bm p - U([d])||_1 \right]=O_{\delta}(\frac{s}{d^2}) \ll 1$. This implies, for any $b >0$, 
		\begin{align*}
		\P\left( ||\bm p  - U([d])||_1 \leq \E\left[||\bm p  - U([d])||_1 \right] - b \right) \leq \frac{\Var\left[||\bm p  - U([d])||_1 \right]}{b^2} \ll 1. 
		\end{align*}
		Therefore, since $b >0$ is arbitrary, with probability going to 1,  
		\begin{align*}
		||\bm p  - U([d])||_1 \geq \E[||\bm p  - U([d])||_1]  & = \frac{2(1-\delta)}{d} (s- \lfloor \delta s \rfloor )  \left(1 - \frac{t}{s} \right)  \nonumber \\ 
		& \geq \frac{2(1-\delta)^3s}{d} . 
		\end{align*}
		To complete the proof recall \eqref{eq:barG} and note that if $\bm p  \in \cG$, then $||\bm p  - U([d])||_0=s$.  
	\end{proof}

	
The lemma above shows that the prior $\pi_n$ induced on $\cP([d])$ by the random probability distribution $\bm p$ constructed above is, in fact, supported on  $\cP(U([d]), s, \varepsilon)$ with high probability, for any  $\varepsilon$ such that 
 $\limsup \frac{d\varepsilon}{s} \leq 2$. Hence, by \cite[Chapter 2]{ingster2012nonparametric} it suffices to analyze the second moment of the $\pi_n$-integrated likelihood ratio. Towards this define, 
	\begin{align}
	L(\bm p)  := \frac{\P_{\bm p}(Z_1, Z_2, \ldots, Z_d)}{\P_{H_0}(Z_1, Z_2, \ldots, Z_d)} 	&=L_1(\bm p ) L_{2, 1}(\bm p )  \bm 1\{\bm \eta \in \cG\} + L_1(\bm p ) L_{2, 2}(\bm p )  \bm 1\{\bm \eta \in \cG^c\} , \nonumber 
	\end{align} 
	where 
	\begin{align}
	L_1(\bm p ) & :=  \prod_{j \in S} \left[e^{\frac{n (1-\delta)}{d} (1-\eta_j \frac{s}{t})} \left( \delta + \eta_j \frac{(1-\delta)s}{t}\right)^{Z_{j}}  \right] , \nonumber \\ 
	L_{2, 1}(\bm p ) & :=  \prod_{j \in T}  e^{\frac{ n |\Delta(\bm \eta)|}{ \lfloor \delta s \rfloor}}\left(1 + \frac{ d |\Delta(\bm \eta)|}{n \lfloor \delta s \rfloor} \right)^{Z_j}  \bm 1\{ \Delta(\bm \eta) < 0\} \nonumber \\ 
	& \;\;\;\;\;\;\;\;\;\;\;\;\;\;\;\;\;\;\;\;\;\;\;\;   +  \prod_{j \in T} e^{\frac{ n |\Delta(\bm \eta)|}{ \lfloor \delta s \rfloor}}\left(1 - \frac{ d |\Delta(\bm \eta)|}{n \lfloor \delta s \rfloor} \right)^{Z_j} \bm 1\{ \Delta(\bm \eta) > 0\} , \nonumber \\ 
	L_{2, 2}(\bm p ) & :=  \left(1 +  d r_n \right)^{Z_{d/2+1}}    \left(1 -  d r_n \right)^{Z_{d/2+2}} . \nonumber 
	\end{align} 
	Then the $\pi_n$-integrated likelihood becomes 
	\begin{align*} 
	L_{\pi_n} := \E_{\bm p } [L(\bm p )] = L_{\pi_n, 1} + L_{\pi_n, 2}, 
	\end{align*}
	where 
	\begin{align}\label{eq:L_H_II}  
	L_{\pi_n, 1} := \E_{\bm p } \left[ L_1(\bm p ) L_{2, 1}(\bm p )  \bm 1\{\bm \eta \in \cG\}  \right] \quad \text{and} \quad L_{\pi_n, 2} := \E_{\bm p } \left[ L_1(\bm p ) L_{2, 2}(\bm p )  \bm 1\{\bm \eta \in \cG^c\} \right] . 
	\end{align}
	The following result bounds the second moments of $L_{\pi_n, 1}$  and $L_{\pi_n, 2}$:

	\begin{lem}\label{lm:L_H12} For $L_{\pi_n, 1}$ and $L_{\pi_n, 2}$ as defined above in \eqref{eq:L_H_II} and under the conditions of Proposition \ref{ppn:threshold_lb_II}, the following hold:  
		\begin{enumerate} 
			
			\item[(a)] $\E_{H_0} \left[L_{\pi_n, 1}^2 \right] \leq 1+ o(1)$. 
			
			\item[(b)] $\E_{H_0} \left[L_{\pi_n, 2}^2 \right]= o(1)$. 
			
		\end{enumerate}
	\end{lem}

The above lemma implies $\E_{H_0} \left[L_{\pi_n}^2 \right] \leq 1+ o(1)$, under the conditions of Proposition \ref{ppn:threshold_lb_II} . Therefore, by \eqref{eq:R_lb}, for any test function $T$, $\overline \cR_{n, d}(T, s, \varepsilon) \rightarrow 1$, under the conditions of Proposition \ref{ppn:threshold_lb_II} This completes the proof of Proposition \ref{ppn:threshold_lb_II}. \\

	\noindent{\it Proof of Lemma} \ref{lm:L_H12} (a): Recalling \eqref{eq:L_H_II} note that,  
	\begin{align}
	L_{\pi_n, 1} & = \E_{\bm p } \left[ L_1(\bm p ) L_{2, 1}(\bm p )  \bm 1\{\bm \eta \in \cG\}  \right] \nonumber \\ 
	& = \E_{(S, T, \bm \eta)} \left[ L_1(S, \bm \eta)  \left\{L_{2, 1}^+(T, \bm \eta) + L_{2, 1}^-(T, \bm \eta) \right\} \bm 1\{\bm \eta \in \cG\} \right], \nonumber 
	\end{align}
	where the expectation is over the randomness of $\{\eta_j: 1 \leq j \leq s-\lfloor \delta s \rfloor\}$, the set $S$, and the set $T$, and 
	\begin{align}
	L_1(S, \bm \eta) & : = \prod_{j \in S} e^{\frac{n (1-\delta)}{d} (1-\eta_j \frac{s}{t})} \left( \delta + \eta_j \frac{(1-\delta)s}{t}\right)^{Z_{j}} \label{eq:L1_T_I} \\ 
	L_{2, 1}^+(T, \bm \eta)  & :=  \prod_{j \in T}  e^{-\frac{ n |\Delta(\bm \eta)|}{ \lfloor \delta s \rfloor}}\left(1 + \frac{ d |\Delta(\bm \eta)|}{\lfloor \delta s \rfloor} \right)^{Z_j} \bm 1\{ \Delta(\bm \eta) < 0\} \label{eq:L21_T_II} \\
	L_{2, 1}^-(T, \bm \eta) & := \prod_{j \in T}  e^{\frac{ n |\Delta(\bm \eta)|}{ \lfloor \delta s \rfloor}}\left(1 - \frac{ d |\Delta(\bm \eta)|}{\lfloor \delta s \rfloor} \right)^{Z_j} \bm 1\{ \Delta(\bm \eta) > 0\} . \label{eq:L21_T_III}
	\end{align}
	Therefore, 
	\begin{align}\label{eq:L1_ST}
	\E_{H_0} L_{\pi_n, 1}^2 & =  \E_{\substack{(S_1, T_1, \bm \eta_1) \\ (S_2, T_2, \bm \eta_2) }} \E_{H_0} \left[L_1(S_1, \bm \eta_1) L_1(S_2, \bm \eta_2)   \left\{ W_1 + W_2 + W_3 + W_4 \right\} \right]
	\end{align}
	where 
	\begin{align}\label{eq:LT12}
	W_1&=L_{2, 1}^+(T_1, \bm \eta_1) L_{2, 1}^+(T_2, \bm \eta_2) \bm 1\{\bm \eta_1, \bm \eta_2  \in \cG\}, \nonumber \\ 
	W_2& =L_{2, 1}^+(T_1, \bm \eta_1) L_{2, 1}^-(T_2, \bm \eta_2) \bm 1\{\bm \eta_1, \bm \eta_2  \in \cG\},  \nonumber \\ 
	W_3&=L_{2, 1}^-(T_1, \bm \eta_1) L_{2, 1}^+(T_2, \bm \eta_2) \bm 1\{\bm \eta_1, \bm \eta_2  \in \cG\},  \nonumber \\  
	W_4&=L_{2, 1}^-(T_1, \bm \eta_1) L_{2, 1}^-(T_2, \bm \eta_2) \bm 1\{\bm \eta_1, \bm \eta_2  \in \cG\}. 
	\end{align} 
	
	\begin{lem}\label{lm:L1_I} Let $W_1, W_2, W_3, W_4$ be as defined in \eqref{eq:LT12}. Then, for $1 \leq i \leq 4$,   
		\begin{align}\label{eq:L21W_I}
		\E_{T_1, T_2}[W_i|S_1, S_2, \bm \eta_1, \bm \eta_2] \leq (1+o(1)) \chi_i, 
		\end{align} 
		where $\chi_1:= \bm 1\{\bm \eta_1, \bm \eta_2  \in \cG, \Delta(\bm \eta_1) <0, \Delta(\bm \eta_2) <0  \}$, $\chi_2:= \bm 1\{\bm \eta_1, \bm \eta_2  \in \cG, \Delta(\bm \eta_1) <0, \Delta(\bm \eta_2)> 0  \}$,  $\chi_3:= \bm 1\{\bm \eta_1, \bm \eta_2  \in \cG, \Delta(\bm \eta_1) > 0, \Delta(\bm \eta_2) <0  \}$, and $\chi_4:= \bm 1\{\bm \eta_1, \bm \eta_2  \in \cG, \Delta(\bm \eta_1) > 0, \Delta(\bm \eta_2) > 0  \}$.  As a consequence, 
		\begin{align}\label{eq:L21W_II}
		\E_{T_1, T_2}[W_1+W_2+W_3+W_4|S_1, S_2, \bm \eta_1, \bm \eta_2] \leq 1 + o(1). 
		\end{align}
	\end{lem}
	
	\begin{proof} To begin with consider $W_1$. Denote by $\cF$ the sigma algebra generated by $(S_1, S_2, \bm \eta_1, \bm \eta_2)$. Recalling \eqref{eq:L21_T_II} and \eqref{eq:L21_T_III} and the moment generating function of the $\dPois(n/d)$ distribution now gives, 
		\begin{align}\label{eq:LT1T2}
		& \E_{T_1, T_2} [W_1|\cF] \nonumber \\ 
		& = \E_{T_1, T_2} \left[\E_{H_0} \left\{L_{2, 1}^+(T_1, \bm \eta_1) L_{2, 1}^+(T_2, \bm \eta_2) \right\}  \Big| \cF \right] \chi_1 \nonumber \\ 
		& = \E_{T_1, T_2} \left[\E_{H_0} \left\{ \prod_{j \in T_1 \cap T_2} e^{-\frac{ n (|\Delta(\bm \eta_1)| + |\Delta(\bm \eta_2)|)}{ \lfloor \delta s \rfloor}}\left\{\left(1 + \frac{ d |\Delta(\bm \eta_1)|}{\lfloor \delta s \rfloor} \right) \left(1 + \frac{ d |\Delta(\bm \eta_2)|}{\lfloor \delta s \rfloor} \right) \right\}^{Z_j}   \right\}  \Big| \cF \right] \chi_1 \nonumber \\ 
		& = \E_{T_1, T_2} \left[ \prod_{j \in T_1 \cap T_2} e^{-\frac{ n (|\Delta(\bm \eta_1)| + |\Delta(\bm \eta_2)|)}{ \lfloor \delta s \rfloor}} e^{ \frac{n(|\Delta(\bm \eta_1)|+|\Delta(\bm \eta_2)|)}{\lfloor \delta s \rfloor}  + \frac{n d |\Delta(\bm \eta_1)| |\Delta(\bm \eta_2)|}{ \lfloor \delta s \rfloor^2}}    \Big| \cF \right] \chi_1 \nonumber \\ 
		& = \E_{T_1, T_2} \left[  e^{  \frac{n d |\Delta(\bm \eta_1)| |\Delta(\bm \eta_2)| |T_1 \cap T_2|}{ \lfloor \delta s \rfloor^2}}  \Big| \cF \right] \chi_1. 
		\end{align} 
		Now, using the fact that $|T_1 \cap T_2| \big| \cF \sim \mathrm{Hypergeometric}(d/2, \lfloor \delta s \rfloor, \lfloor \delta s \rfloor)$, which is dominated by the $\dBin(\lfloor \delta s \rfloor, \frac{2 \lfloor \delta s \rfloor}{d} )$ distribution in convex ordering \cite[Proposition 20.6]{aldous1985exchangeability}, it follows that,  
		\begin{align}
		\E_{T_1, T_2}[W_1|\cF] & =\left(1- \frac{2 \lfloor \delta s \rfloor }{d} + \frac{2 \lfloor \delta s \rfloor }{d}  e^{  \frac{n d |\Delta(\bm \eta_1)| |\Delta(\bm \eta_2)| }{ \lfloor \delta s \rfloor^2}}  \right)^{\lfloor \delta s \rfloor} \chi_1 \nonumber \\ 
		& \leq \exp\left\{  \frac{2 \lfloor \delta s \rfloor^2 }{d} \left(e^{  \frac{n d |\Delta(\bm \eta_1)| |\Delta(\bm \eta_2)| }{ \lfloor \delta s \rfloor^2}}  - 1 \right) \right\} \chi_1 \nonumber \\ 
		& \leq \exp\left\{  \frac{2 \lfloor \delta s \rfloor^2 }{d} \left(e^{  \frac{n s^2  \gamma_d^2 }{d  t \lfloor \delta s \rfloor^2}}  - 1 \right) \right\} \chi_1 \tag*{(since $\max\{|\Delta(\bm \eta_1)|, |\Delta(\bm \eta_2)|\} \leq \frac{\gamma_d s}{d \sqrt t}$ on $\chi_1$)} \nonumber \\ 
		\label{eq:LT12_F_I}&=\exp\left\{ O\left(   \frac{n s  \gamma_d^2 }{d^2 } \right) \right\} \chi_1 \\ 
		\label{eq:LT12_F_II} & = (1+o(1)) \chi_1.
		\end{align}
		For \eqref{eq:LT12_F_I} we use $e^x-1 \leq x + O(x^2)$ as $x \rightarrow 0$ and $\frac{n t_{s, d}^2}{dt} \ll 1$. This can be shown by recalling $\gamma_d^2 \ll \frac{\sqrt s}{\log d}$ as follows: 
		\begin{itemize} 
			
			\item If $\alpha > \frac{1}{2}$ and $n \lesssim d \log d$, $\frac{n t_{s, d}^2}{dt} \lesssim \frac{n t_{s, d}^2}{ds} \ll \frac{n}{d \log d} \lesssim 1$.  
			
			\item If $\alpha \leq \frac{1}{2}$ and $n \ll d^{\frac{1}{2}+\alpha}$, $\frac{n t_{s, d}^2}{dt} \lesssim \frac{n t_{s, d}^2}{ds} \ll \frac{n}{d \log d} \ll \frac{1}{d^{\frac{1}{2}-\alpha} \log d} \ll 1$. 
			
		\end{itemize}
		Similarly, for \eqref{eq:LT12_F_II} we use $e^x-1 \leq x + O(x^2)$ as $x \rightarrow 0$  and  $\frac{n s  \gamma_d^2 }{d^2 } \ll 1$. To see this, we use  $\gamma_d^2 \ll \frac{\sqrt s}{\log d}$ and note the following: 
		\begin{itemize} 
			
			\item If $\alpha \leq \frac{1}{2}$ and $n \ll d^{\frac{1}{2}+\alpha}$, $\frac{ns \gamma_d^2}{d^2} = \frac{n \gamma_d^2}{d^{1+\alpha}} \ll \frac{\sqrt s}{\sqrt d \log d} \ll 1$. 
			
			\item If $\alpha > \frac{1}{2}$ and $n \lesssim d \log d$, then $\frac{ns \gamma_d^2}{d^2} = \frac{n \gamma_d^2}{d^{1+\alpha}} \lesssim \frac{\sqrt s}{d^\alpha} = d^{\frac{1}{2} - \frac{3 \alpha}{2}}\ll 1$. 
		\end{itemize}

		{Next, consider $W_2$. Then as in \eqref{eq:LT1T2} above,   
			\begin{align}
			\E_{T_1, T_2} & [W_2|\cF] \nonumber \\ 
			& = \E_{T_1, T_2} \left[\E_{H_0} \left\{L_{2, 1}^+(T_1, \bm \eta_1) L_{2, 1}^-(T_2, \bm \eta_2) \right\}  \Big| \cF \right] \chi_2 \nonumber \\ 
			& = \E_{T_1, T_2} \left[\E_{H_0} \left\{ \prod_{j \in T_1 \cap T_2} e^{-\frac{ n (|\Delta(\bm \eta_1)| - |\Delta(\bm \eta_2)|)}{ \lfloor \delta s \rfloor}}\left\{\left(1 + \frac{ d |\Delta(\bm \eta_1)|}{\lfloor \delta s \rfloor} \right) \left(1 - \frac{ d |\Delta(\bm \eta_2)|}{\lfloor \delta s \rfloor} \right) \right\}^{Z_j}   \right\}  \Big| \cF \right] \chi_2 \nonumber \\ 
			& = \E_{T_1, T_2} \left[ \prod_{j \in T_1 \cap T_2} e^{-\frac{ n (|\Delta(\bm \eta_1)| - |\Delta(\bm \eta_2)|)}{ \lfloor \delta s \rfloor}} e^{ \frac{n(|\Delta(\bm \eta_1)|-|\Delta(\bm \eta_2)|)}{\lfloor \delta s \rfloor}  - \frac{n d |\Delta(\bm \eta_1)| |\Delta(\bm \eta_2)|}{ \lfloor \delta s \rfloor^2}}    \Big| \cF \right] \chi_2  \nonumber \\ 
			& \leq \chi_2. \nonumber  
			\end{align} }

		The result in \eqref{eq:L21W_I} for $W_3$ and $W_4$ follows similarly. This implies the result in \eqref{eq:L21W_II}, since $\chi_1+\chi_2+\chi_3+\chi_4 \leq 1$. \end{proof}

	\begin{lem}\label{lm:L1_II} For $L_1$ as defined in \eqref{eq:L1_T_I} and the conditions of Proposition \ref{ppn:threshold_lb_II}, 
		$$\E_{\substack{(S_1, \bm \eta_1) \\ (S_2, \bm \eta_2) }} \E_{H_0} \left[L_1(S_1, \bm \eta_1) L_1(S_2, \bm \eta_2)  \right] \leq 1+ o(1).$$
	\end{lem}
	
	\begin{proof} To begin with note that 
		\begin{align}\label{eq:L1_lb_ts}
		\E_{(S, \bm \eta)} \left[ L_1(S, \bm \eta) \right] & =  \E_{(S, \bm \eta)} \prod_{j \in S} \left[e^{\frac{n (1-\delta)}{d} (1-\eta_j \frac{s}{t})} \left( \delta + \eta_j \frac{(1-\delta)s}{t}\right)^{Z_{j}}  \right]. 
		\end{align} 
		Denote $\bar \delta=1-\delta$ and $\bar \eta_j=\eta_j \frac{s}{t} -1$. Then using \eqref{eq:L1_lb_ts}, 
		\begin{align}
		& \E_{\substack{(S_1, \bm \eta_1) \\ (S_2, \bm \eta_2) }} \E_{H_0} \left[L_1(S_1, \bm \eta_1) L_1(S_2, \bm \eta_2)  \right]  \nonumber \\ 
		&= \E_{\substack{(S_1, \bm \eta_1)\\(S_2, \bm \eta_2)}} \E_{H_0}\left[  \prod_{j \in S_1 \cap S_2} e^{\frac{n \bar \delta }{d} (2-(\eta_j+\eta_j') \frac{s}{t})} \left(\left( 1-\bar \delta + \eta_j \frac{\bar \delta s}{t}\right) \left( 1-\bar \delta + \eta_j' \frac{\bar \delta s}{t}\right)\right)^{Z_{j}}  \right]
		\nonumber \\ 
		&= \E_{\substack{(S_1, \bm \eta_1)\\(S_2, \bm \eta_2)}} \E_{H_0}\left[  \prod_{j \in S_1 \cap S_2} e^{\frac{n \bar \delta }{d} (2-(\eta_j+\eta_j') \frac{s}{t})} \left( ( 1 + \bar \delta \bar \eta_j )  ( 1 + \bar \delta \bar \eta_j' ) \right)^{Z_{j}}  \right]
		\nonumber \\ 
		&= \E_{\substack{(S_1, \bm \eta_1)\\(S_2, \bm \eta_2)}} \left[  \prod_{j \in S_1 \cap S_2} e^{-\frac{n \bar \delta }{d} (\bar \eta_j+\bar \eta_j')  } e^{\frac{n}{d}( (\bar \eta_j + \bar \eta_j')\bar \delta + \bar \eta_j \bar \eta_j' \bar \delta^2)} \right] \tag*{(using \eqref{eq:exptheta}) }\nonumber \\ 
		\nonumber \\ 
		&= \E_{\substack{(S_1, \bm \eta_1)\\(S_2, \bm \eta_2)}} \left[  \prod_{j \in S_1 \cap S_2}   e^{\frac{n \bar \delta^2}{d}  \bar \eta_j \bar \eta_j'  } \right] \nonumber \\ 
		\nonumber \\ 
		&= \E_{\substack{(S_1, \bm \eta_1)\\(S_2, \bm \eta_2)}} \left[ e^{\frac{n \bar \delta^2}{d}  (\frac{s}{t} -1)^2 |(S_1 \cap S_2)_{11}| } e^{-\frac{n \bar \delta^2}{d}  (\frac{s}{t} -1) |(S_1 \cap S_2)_{01}| } e^{\frac{n \bar \delta^2}{d}  |(S_1 \cap S_2)_{00}| } \right] \nonumber \\ 
		&= \E_{\substack{(S_1, \bm \eta_1)\\(S_2, \bm \eta_2)}} \left[ e^{-\frac{n \bar \delta^2}{d}  (\frac{s}{t} -1) |(S_1 \cap S_2)| } e^{\frac{n \bar \delta^2}{d} \left((\frac{s}{t} -1)^2 + (\frac{s}{t} -1) \right) |(S_1 \cap S_2)_{11}| }  e^{\frac{n \bar \delta^2}{d}  \left( (\frac{s}{t} -1)  + 1 \right) |(S_1 \cap S_2)_{00}| } \right] \nonumber \\ 
		&= \E_{\substack{(S_1, \bm \eta_1)\\(S_2, \bm \eta_2)}} \left[ e^{-\frac{n \bar \delta^2}{d}  (\frac{s}{t} -1) |(S_1 \cap S_2)| } e^{\frac{n s \bar \delta^2}{dt} (\frac{s}{t} -1) |(S_1 \cap S_2)_{11}| }  e^{\frac{n s \bar \delta^2}{dt}  |(S_1 \cap S_2)_{00}| } \right],
		\label{eq:L2_lb_I}
		\end{align}
		where $(S_1\cap S_2)_{11}=\{r: \eta_r=\eta_r'=1\}$, $(S_1\cap S_2)_{00}=\{r: \eta_r =  \eta_r'=0\}$, and  $(S_1\cap S_2)_{01}=\{r: \eta_r \ne \eta_r' \}$. Now, using $|(S_1 \cap  S_2)_{11}| \Big| S_1  \cap  S_2 \sim \dBin(|S_1  \cap  S_2|, \frac{t^2}{s^2})$ and $|(S_1 \cap  S_2)_{00}| \Big| S_1  \cap  S_2 \sim \dBin(|S_1  \cap  S_2|, (1-\frac{t}{s})^2)$, it follows that 
		\begin{align}
		\E_{H_0} \left[ e^{\frac{n s \bar \delta^2}{dt} (\frac{s}{t} -1) |(S_1 \cap S_2)_{11}| }  \Big| S_1  \cap  S_2\right] & = \left[\left(1-\frac{t^2}{s^2} \right) + \frac{t^2}{s^2}  e^{\frac{n s \bar \delta^2}{dt} (\frac{s}{t} -1) } \right]^{|S_1 \cap S_2| } \nonumber \\ 
	\label{eq:S12_I}	& = \exp\left\{|S_1\cap S_2| \log\left(1-\frac{t^2}{s^2}  \left[ 1- e^{\frac{n s \bar \delta^2}{dt} (\frac{s}{t} -1) } \right]\right)\right\}  \\ 
		& \leq \exp\left\{- (1+o(1)) |S_1\cap S_2| \frac{t^2}{s^2}  \left[ 1- e^{\frac{n s \bar \delta^2}{dt} \left(\frac{s}{t} -1 \right) } \right] \right\}.  \nonumber
		\end{align} 
		Similarly,  
		\begin{align}\label{eq:S12_II}  
		\E_{H_0} \left[ e^{\frac{n s \bar \delta^2}{dt}  |(S_1 \cap S_2)_{00}| }  \Big| S_1  \cap  S_2\right] & = \left[ 1- \left(1-\frac{t}{s}\right)^2  + \left(1-\frac{t}{s}\right)^2  e^{\frac{n s \bar \delta^2}{d t}} \right]^{|S_1 \cap S_2| } \\ 
				& \leq \exp\left\{-(1+o(1)) |S_1\cap S_2| \left(1-\frac{t}{s} \right)^2  \left[ 1- e^{\frac{n s \bar \delta^2}{dt}  } \right] \right\}. \nonumber 
		\end{align}
By the negative association property of the multinomial distribution \cite[Section 3]{multinomial}, it follows that $|(S_1 \cap  S_2)_{11}| \big| S_1  \cap  S_2$ and $|(S_1 \cap  S_2)_{00}| \big| S_1  \cap  S_2$ have the negative association property. Using this and recalling  \eqref{eq:L2_lb_I} gives, 
		\begin{align}\label{eq:L2_lb_II}
		& \E_{H_0} \left[\E_{\bm \eta_1, \bm \eta_2} \left[L_1(S_1, \bm \eta_1) L_1(S_2, \bm \eta_2)  \Big | S_1 \cap S_2 \right] \right]  \nonumber \\
		& \leq  \exp\left\{ (1+o(1))  |S_1\cap S_2| \left( - \frac{n \bar \delta^3}{\delta d}  + \delta^2  \left[ e^{\frac{n  \bar \delta^3}{\delta^2 d} } - 1 \right] + \bar \delta^2  \left[ e^{\frac{n  \bar \delta^2}{\delta d} } -1 \right] \right) \right\} .  
		\end{align}
		Now consider two cases: 
		
		\begin{itemize} 
			
			\item $\alpha \leq \frac{1}{2}$ and $n \ll d^{\frac{1}{2}+\alpha}$: This implies $n \ll d$. Hence,  
			using $e^x-1 \leq x +O(x^2)$, when $x \rightarrow 0$, in \eqref{eq:L2_lb_II} and the Hypergeometric-Binomial convex ordering argument \cite[Proposotion 20.6]{aldous1985exchangeability} gives 
			\begin{align}
			\E_{\substack{(S_1, \bm \eta_1) \\ (S_2, \bm \eta_2) }} \E_{H_0} \left[L_1(S_1, \bm \eta_1) L_1(S_2, \bm \eta_2)  \right]  & \leq \E_{S_1, S_2}e^{|S_1\cap S_2| O\left(\frac{n^2}{d^2}\right)} \nonumber \\ 
			&\leq \left(1-\frac{s}{d} + \frac{s}{d} e^{ O\left(\frac{n^2}{d^2}\right)} \right)^s \nonumber \\
			& = \exp\left\{s \log \left(1-\frac{s}{d} \left[1 - e^{ O\left(\frac{n^2}{d^2}\right)} \right] \right)\right\} \nonumber  \\ 
			&\leq \exp\left\{-\frac{s^2}{d} \left[1 - e^{ O\left(\frac{n^2}{d^2}\right)} \right] \label{eq:S1S2} \right\} \\ 
		& \leq e^{O\left(\frac{n^2 s^2}{d^3}\right) }=1+o(1), \nonumber 
			\end{align}
			where the last step uses $\frac{n^2 s^2}{d^3} \ll 1$, which follows from $n \ll d^{\frac{1}{2}+\alpha}$. 
			
			\item $\alpha > \frac{1}{2}$ and $n \ll d \log d$: To begin with suppose $n \ll d$. Then by arguments as in \eqref{eq:S1S2},  
		\begin{align*}
		\E_{\substack{(S_1, \bm \eta_1) \\ (S_2, \bm \eta_2) }} \E_{H_0} \left[L_1(S_1, \bm \eta_1) L_1(S_2, \bm \eta_2)  \right] & \leq \exp\left\{\frac{s^2}{d} \left[ e^{ O\left(\frac{n^2}{d^2}\right)} - 1 \right] \right\} \nonumber \\ 
		& \leq e^{O\left(\frac{n^2 s^2}{d^3}\right) }=1+o(1), \nonumber 
			\end{align*}
since $n \ll d$ and $s^2 \ll d$ for $\alpha > \frac{1}{2}$. Now, suppose $d \lesssim n \ll d \log d$. Then from \eqref{eq:S12_I}, 
\begin{align}\label{eq:S12dlogd_I}
		\E_{H_0} \left[ e^{\frac{n s \bar \delta^2}{dt}  |(S_1 \cap S_2)_{11}| }  \Big| S_1  \cap  S_2\right] & \leq  \exp\left\{|S_1\cap S_2| \left(\frac{n s \bar \delta^2}{d t} \left(\frac{s}{t} -1\right) + O(1) \right) \right\}. 
\end{align}		
Similarly, from \eqref{eq:S12_II}, 
\begin{align}\label{eq:S12dlogd_II}
		\E_{H_0} \left[ e^{\frac{n s \bar \delta^2}{dt}  |(S_1 \cap S_2)_{00}| }  \Big| S_1  \cap  S_2\right] & \leq  \exp\left\{|S_1\cap S_2| \left(\frac{n s \bar \delta^2}{d t} \left(1-\frac{t}{s} \right)^2 + O(1) \right) \right\}. 
\end{align} 
Using \eqref{eq:S12dlogd_I} and \eqref{eq:S12dlogd_II} in \eqref{eq:L2_lb_I} now gives, 
		\begin{align*}
		\E_{H_0} \left[\E_{\bm \eta_1, \bm \eta_2} \left[L_1(S_1, \bm \eta_1) L_1(S_2, \bm \eta_2)  \Big | S_1 \bigcap S_2 \right] \right]    \leq e^{|S_1\cap S_2| O(1)} .  
		\end{align*}
This implies, by the Hypergeometric-Binomial convex ordering argument \cite[Proposition 20.6]{aldous1985exchangeability}, that 
		\begin{align}
			\E_{\substack{(S_1, \bm \eta_1) \\ (S_2, \bm \eta_2) }} \E_{H_0} \left[L_1(S_1, \bm \eta_1) L_1(S_2, \bm \eta_2)  \right]  & \leq \E_{S_1, S_2}e^{|S_1\cap S_2| O(1)} \nonumber \\ 
			&\leq \left(1-\frac{s}{d} + \frac{s}{d} e^{ O(1)} \right)^s \nonumber \\
			& \leq e^{O\left(\frac{s^2}{d}\right) }=1+o(1), \nonumber 
		\end{align}
since $s^2 \ll d$ for $\alpha > \frac{1}{2}$.			
		\end{itemize} 		
This completes the proof of Lemma \ref{lm:L1_II}. 
	\end{proof}

	The proof of Lemma \ref{lm:L_H12} (a) now can be completed using the above two lemmas as follows: Recall from \eqref{eq:L1_ST}, 
	\begin{align}
	\E_{H_0} L_{\pi_n, 1}^2 & = \E_{\substack{(S_1, \bm \eta_1) \\ (S_2, \bm \eta_2) }} \E [\E_{T_1, T_2} [ L_{\pi_n, 1}^2 | S_1, S_2, \bm \eta_1, \bm \eta_2 ] ] \nonumber \\ 
	& = \E_{\substack{(S_1, \bm \eta_1) \\ (S_2, \bm \eta_2) }}  \left[L_1(S_1, \bm \eta_1) L_1(S_2, \bm \eta_2)  \E_{T_1, T_2} [\left\{ W_1 + W_2 + W_3 + W_4 \right\} | S_1, S_2, \bm \eta_1, \bm \eta_2 ]\right] \nonumber \\ 
	& \leq (1+o(1)) \E_{\substack{(S_1, \bm \eta_1) \\ (S_2, \bm \eta_2) }}  \left[L_1(S_1, \bm \eta_1) L_1(S_2, \bm \eta_2)   \right] \tag*{(by Lemma \ref{lm:L1_I})} \nonumber \\ 
	& \leq (1+o(1)),  \nonumber 
	\end{align}
	where the last step uses Lemma \ref{lm:L1_II}. This completes the proof of Lemma \ref{lm:L_H12} (a). \\ 
	
	\noindent{\it Proof of Lemma} \ref{lm:L_H12} (b): Note that 
	Recall that 
	$$L_{\pi_n, 2} = \E_{(S, \bm \eta)} \left[ L_1(S, \bm \eta)   \left(1 +  d r_n \right)^{Z_{d/2+1}}    \left(1 -  d r_n \right)^{Z_{d/2+2}}  \bm 1\{\bm \eta \in \cG^c\} \right].$$
	This implies, 
	\begin{align}
 &	\E_{H_0} [L_{\pi_n, 2}^2] \nonumber \\ 
	& =  \E_{\substack{(S_1, \bm \eta_1) \\ (S_2, \bm \eta_2) }} \E_{H_0} \left[L_1(S_1, \bm \eta_1) L_1(S_2, \bm \eta_2)  \left(1 +  d r_n \right)^{2 Z_{d/2+1}}    \left(1 -  d r_n \right)^{2 Z_{d/2+2} } \bm 1\{\bm \eta_1 \in \cG^c\} \bm 1\{\bm \eta_2 \in \cG^c\} \right] \nonumber \\ 
	& = e^{2n d r_n^2} \E_{\substack{(S_1, \bm \eta_1)\\(S_2, \bm \eta_2)}} \left[   \prod_{j \in S_1 \cap S_2}   e^{\frac{n \bar \delta^2}{d}  \bar \eta_j \bar \eta_j'  }    \bm 1\{\bm \eta_1 \in \cG^c\} \bm 1\{\bm \eta_2 \in \cG^c\}  \right] \nonumber \\ 
	& \leq e^{2n d r_n^2} \left( \E_{\substack{(S_1, \bm \eta_1)\\(S_2, \bm \eta_2)}} \left[     e^{\frac{2 n \bar \delta^2}{d} \sum_{j \in S_1 \cap S_2} \bar \eta_j \bar \eta_j'  } \right] \right)^{\frac{1}{2}}   \P(\bm \eta_1 \in \cG^c ) \nonumber \\ 
	&= (1+o(1)) e^{2n d r_n^2} \P(\bm \eta_1 \in \cG^c ) = o(1), \nonumber 
	\end{align}
	using arguments as in Lemma \ref{lm:L1_II}, $\P(\bm \eta_1 \in \cG^c ) = o(1)$ (recall \eqref{eq:barG}), and $e^{2n d r_n^2}=1+o(1)$ (since $r_n \ll 1/\sqrt{nd}$).  \hfill $\Box$

\small

\medskip

\small{\subsection*{Acknowledgment}  B. B. Bhattacharya was supported by NSF CAREER grant DMS 2046393, a Sloan Research Fellowship, and Wharton Dean's Research Fund. }

\bibliographystyle{abbrvnat}
\bibliography{biblio_sparse_uniformity}

\normalsize 	
		
	\appendix
	
	\section{Technical Lemmas} 
\label{sec:lm_pf}
In this section we collect the proofs of various technical lemmas. We begin with the proof of Proposition \ref{ppn:L01}. \\ 
	
%

\noindent {\it Proof of Proposition \ref{ppn:L01}}:  	

We begin by proving the upper bound on $\varepsilon_{\max}$. For this, without loss of generality assume that the first $s$ coordinates of $U([d])$ and $\bm p$ are different. Define $\Delta_i := \frac{1}{d}-p_i$, which is non-zero, for $1 \leq i \leq s$,  and zero, for $s+1 \leq i \leq n$. Note that if $\Delta_i > 0$, then $p_i=\frac{1}{d}-\Delta_i  >0$ which means $\Delta_i  < \frac{1}{d}$. Now, using $\sum_{i=1}^s \Delta_i=\sum_{i=1}^d \Delta_i = 0$ gives, 
		$$||\bm p- U([d])||_1=\sum_{i=1}^s |\Delta_i| =  \sum_{i=1}^s \Delta_i \bm 1\{\Delta_i > 0\} - \sum_{i=1}^s \Delta_i \bm 1\{\Delta_i < 0\} = 2 \sum_{i=1}^s \Delta_i \bm 1\{\Delta_i > 0\}  <  \frac{2 s}{d}.$$
This implies, $\limsup \frac{d \varepsilon_{\max}}{s} \leq 2$. 
		
To prove a matching lower bound, consider $\bm p = (p_1, p_2, \ldots, p_d) \in U([d])$ as follows: 
$$p_j= 
	\left\{
	\begin{array}{cc}
	\frac{s}{d}   &  \text{for }  j = 1, \\ 
	0  &  \text{for } 2 \leq j \leq s, \\
	\frac{1}{d} & \text{otherwise}. 
	\end{array}
	\right.
$$ 	
For this choice of $\bm p$, it is easy to check that $\sum_{j=1}^d p_j=1$, $||\bm p - U([d])||_0$, and $||\bm p - U([d])||_1= \frac{2(s-1)}{d}$. This shows, $\liminf \frac{d \varepsilon_{\max}}{s} \geq 2$. \qed

\medskip
	
Next we prove the following observation regarding the tails of the normal distribution. For this recall that $\bar \varepsilon_j := |\frac{1}{d} - p_j| \sqrt{nd/(2 \log d)}$, for $j \in [d]$.

		\begin{obs}\label{obs:normal_tail} Let $\lambda(d)=\sqrt{2 \log d}$ and $r > 0$ be fixed. Then, if $n \gg d \log^3 d$, 
			\begin{align}\label{eq:normal_tail}
			\sup_{j \in [d]} \left| \frac{\bar \Phi\left(\lambda(d) \sqrt{1/(d p_j)}  \left(\sqrt r \pm \bar \varepsilon_j \right) \right)}{ \bar \Phi\left(\lambda(d)  \left(\sqrt r \pm \bar \varepsilon_j \right) \right)} - 1 \right| = o(1), 
			\end{align}
			where $\bar \varepsilon_j$ is as defined above. 
		\end{obs}
		
		\begin{proof} We prove the result in the positive case. The negative case can be done similarly. For $j \in [d]$, let $\beta_{j}:=\sqrt{1/(d p_j)} (\sqrt r + \bar \varepsilon_j )$ and  $\gamma_j:= (\sqrt r + \bar \varepsilon_j )$. Note that 
			$$\sup_{j \in [d]} \left|\frac{\gamma_j}{\beta_j}-1 \right|=O\left(\sqrt{\frac{\log d}{nd}}\right).$$ 
			Then using the inequalities: $\frac{x}{1+x^2} \phi(x) \leq \Phi(x) \leq \frac{1}{x} \phi(x)$, for all $x > 0$ gives, for every $j\in [d]$
			$$\frac{\bar \Phi(\gamma_j  \lambda(d) )}{\bar \Phi( \beta_j \lambda(d) )} =  (1+o(1))  e^{ (\beta_j^2 -\gamma_j^2) \log d},$$
			where the $o(1)$-term goes to zero uniformly over $j \in [d]$. This implies the result in \eqref{eq:normal_tail}, since $\sup_{j \in [d]} |\beta_j^2-\gamma_j^2| \log d =O(\sqrt{\frac{d\log^3 d}{n}})=o(1)$. 
		\end{proof}

\end{document}